\documentclass[10pt,double,reqno]{amsart}

\usepackage{graphicx}
\usepackage{amssymb}
\usepackage{epsfig}
\usepackage{amsmath}
\usepackage{amsthm}
\usepackage{color}

\newtheorem{theo}{Theorem}[section]
\newtheorem{defin}[theo]{Definition}

\newtheorem{lemm}[theo]{Lemma}
\newtheorem{rem}[theo]{Remark}

\numberwithin{equation}{section}

\newcommand{\al}{\alpha}
\newcommand{\be}{\beta}

\newcommand{\Ga}{\Gamma}

\newcommand{\om}{\omega}

\newcommand{\te}{\theta}
\newcommand{\De}{\Delta}

\newcommand{\pa}{\partial}

\newcommand{\R}{{\mathbb R}^n}

\newcommand{\Rn}{{\mathbb R}^{n-1}}

\newcommand{\na}{\nabla}

\begin{document}
\baselineskip=18pt

\title[]{Solvability of the Initial-Boundary value problem of the    Navier-Stokes equations  
with rough data}

\
\author{Tongkeun Chang}
\address{Department of Mathematics, Yonsei University \\
Seoul, 136-701, South Korea}
\email{chang7357@yonsei.ac.kr}

\author{Bum Ja Jin}
\address{Department of Mathematics, Mokpo National University, Muan-gun 534-729,  South Korea }
\email{bumjajin@hanmail.net}

\thanks{This research was supported by Basic Science Research Program
        through the National Research Foundation of Korea(NRF) funded
        by the Ministry of Science, ICT and future planning(NRF-2014R1A1A3A04049515).}

\begin{abstract}
In this  paper, we study the initial and boundary value problem of the Navier-Stokes equations in the half space.
We prove the  unique existence of weak solution $u\in L^q(\R_+\times (0,T))$  with $\nabla u\in L^{\frac{q}{2}}_{loc}(\R_+\times (0,T))$ for a short time interval
when the  initial data $h\in  {B}_q^{-\frac{2}{q}}(\R_+)$
and the boundary data
 $ g\in L^q(0,T;B^{-\frac{1}{q}}_q(\Rn))+L^q(\Rn;B^{-\frac{1}{2q}}_q(0,T)) $  with normal component $g_n\in L^q(0,T;\dot{B}^{-\frac{1}{q}}_q(\Rn))$,  $n+2<q<\infty$ are given.



\end{abstract}

\maketitle

\section{\bf Introduction}
\setcounter{equation}{0}

Let $\R_+ = \{ x \in \R \, | \, x_n > 0 \}$, $n\geq 2$ and $0 < T < \infty$.
Let us  consider the nonstationary Navier-Stokes equations
\begin{align}\label{maineq2}
\begin{array}{l}\vspace{2mm}
u_t - \De u + \na p =f-\mbox{div}(u\otimes u), \qquad div \, u =0, \mbox{ in }
 \R_+\times (0,T),\\
\hspace{30mm}u|_{t=0}= h, \qquad  u|_{x_n =0} = g.
\end{array}
\end{align}

There are abundant literatures for the study of the Navier-Stokes equations with homogeneous boundary data.
See \cite{amann,cannone,kozono1,sol1} and references therein for   the half space problem.
See also \cite{amann,fabes,giga2,giga3,giga1,iftimie,kato1,kato2, kato,kozono2,lemarie} and the references therein for the problems in other domains such as whole space, a bounded domain, or exterior domain.

Over the past decade, the Navier-Stokes equations with  the nonhomogeneous boundary data have been studied actively.
See \cite{fernandes,amann1,amann2,lewis,voss}  and references therein for the  half space problem. 
%
See also \cite{amann1, amann2,farwig4,farwig6,farwig2,farwig3,grubb1,  grubb3, grubb,kozono2} and the references therein for the problems in other domains such as whole space, a bounded domain, or exterior domain.

In \cite{grubb1,grubb3,grubb,sol1}, the solvabilities of bounded or exterior domain problem  have been studied for a  boundary data in anisotropic space $B^{\al-\frac{1}{q}, \frac{\al}{2}-\frac{1}{2q}}_{q0}(\partial \Omega \times (0,T))$, $\al>\frac{1}{q}$ (with $q>\frac{n+2}{\al+1}$), where $g\in B^{s,\frac{s}{2}}_{q0}(S\times (0,T))$ means the zero extension of $g$ to $S\times (-\infty,T)$ is in $B^{s,\frac{s}{2}}_{q}(S\times (-\infty,T))$. On the other hand, in \cite{fernandes, amann1, amann2,farwig4,farwig6, farwig2, farwig3,lewis,voss} a rough boundary data have been considered.
H. Amann\cite{amann1} showed unique maximal 
solution  $u\in L^r_{loc}(0,T^*, H^{\frac{1}{r}}_q(\Omega))$, $3<q<r<\infty, \frac{1}{r}+\frac{3}{r}\leq 1$ for some maximal time $T^*$ in any domain in ${\mathbb R}^3$ with nonempty compact smooth boundary when a nonzero initial data in $B^{-\frac{1}{r}}_q(\Omega)\cap L^q_{\sigma}(\Omega)$ and  nonzero boundary data in $L^r_{loc}({\mathbb R}_+;W^{-\frac{1}{q}+\frac{1}{r}}_q(\partial \Omega))$ are given.
%
 J.E.Lewis\cite{lewis} showed a global in time existence of solution in $L^p({\mathbb R}_+;L^q(\R_+))$ for small data $h\in L^{r_1}(\R_+)\cap L^{r_2}(\R_+)$ and $g\in L^d({\mathbb R}_+;L^r(\R_+))$ with $r_1,r_2,p,q,r,d<\infty$, $r_1<n<r_2, \frac{n-1}{r}+\frac{2}{d}=1, $ and $\frac{2}{q}+\frac{2}{p}=1$.
 K.A.Voss\cite{voss} showed the existence of a global in time self-similar solution for small data  $h\in \dot{B}^{-\frac{1}{2}}_{6,\infty}({\mathbb R}^3_+)\cap  \dot{B}^{-\frac{1}{4}}_{4,\infty}({\mathbb R}^3_+)$ and $t^{\frac{1}{3}}g(t)\in L^\infty({\mathbb R}_+;L^3({\mathbb R}^2))$ with $g_n=0$.
   M.Fernandes de Almeida  and L.C.F. Ferreira\cite{fernandes} showed the existence of global in time solution in the framework of Morrey space for a small data
   $h\in {\mathcal M}_{p,n-p}(\R_+)$, $t^{\frac{1}{2}-\frac{p-1}{2r}}g\in BC({\mathbb R}_+,{\mathcal M}_{r,n-p}(\Rn))$ and $t^{\frac{1}{2}-\frac{p}{2q}} g_n \in BC({\mathbb R}_+,{\mathcal M}_{\frac{(p-1)q}{p},n-p}(\Rn)),$ $2<p,q<\infty$, $1<r<\infty$.

In particular, R. Farwig, H. Kozono and H. Sohr\cite{farwig2}   showed the local in time existence of a very weak solution $u\in L^s(0,T;L^q(\Omega))$ in an exterior domain when nonzero  initial in $B^{-\frac{2}{s}}_{q,s}$ and nonzero boundary data  in $L^s(0,T;W^{-\frac{1}{q}}_q(\partial\Omega))$ for  $\frac{2}{s}+\frac{3}{q}=1,\, 2<s<\infty, \, 3<q<\infty$ are given (Precisely speaking, in \cite{farwig2} a nonzero divergence is considered).

In this paper, we show the unique existence of     $u\in L^q(\R_+\times (0,T))$ with $\nabla u\in L^{\frac{q}{2}}(\R_+\times (0,T))$ for   the Navier-Stokes equations  \eqref{maineq2} for a small time interval $(0,T)$ with the  initial  $h\in B^{-\frac{2}{q}}_q(\R_+)$ and  the  boundary  data $g\in L^q(0,T;B^{-\frac{1}{q}}_q(\Rn))+L^q(\Rn;B^{-\frac{1}{2q}}_q(0,T))
$ with $g_n\in L^q(\Rn;B^{-\frac{1}{2q}}_q(0,T))$, $q>n+2$.
Our result could be compared  with the one in  \cite{farwig2}. The case $q=r=5$ in \cite{farwig2} coincides with the case $q=5$ in our result, except the fact that
   our result cover larger class for $g'$ (the  tangential component of the boundary data) since $ L^q(0,T;B^{-\frac{1}{q}}_q(\Rn))+L^q(\Rn;B^{-\frac{1}{2q}}_q(0,T)) \supsetneqq L^q(\Rn;B^{-\frac{1}{2q}}_q(0,T))$.

The following  is the main result of this paper.
\begin{theo}
\label{thm3}
Let  $\infty>q>n+2$.
Assume that  $h\in {B}_q^{-\frac{2}{q}}(\R_+)$ with $ \mbox{div}\, h=0$, 
$
 g\in L^q(0,T;B^{-\frac{1}{q}}_q(\Rn))+L^q(\Rn;B^{-\frac{1}{2q}}_q({\mathbb R}_+))$
 with  $g_n\in    L^q({\mathbb R}_+;\dot{B}^{-\frac{1}{q}}_q(\Rn))$.
     Then there is $T^*(0<T^*<\infty) $ so that    the Navier-Stokes equations
\eqref{maineq2}  has a  unique   weak solution
$u\in L^q({\mathbb R}^n_+\times (0,T^*))$  with $\nabla^{} u\in L^{\frac{q}{2}}_{loc}(\R_+\times (0,T^*))$.
%
\end{theo}
The space $L^q(0,T;B^{-\frac{1}{q}}_q(\Rn))+L^q(\Rn;B^{-\frac{1}{2q}}_q(0,T))$ coincides with  anisotropic Besov space ${B}_q^{-\frac{1}{q},-\frac{1}{2q}}(\Rn\times {\mathbb R}_+)$ (see section \ref{notation}).
Our result is optimal in the sense that the spaces for the   initial and the boundary data cannot be enlarged for our solution class.
Our arguments in this paper are based on the elementary estimates of the heat operator and the  Laplace operator. The solution representation in section \ref{decomposition} could be useful to study asymptotic behavior of the solution.


Before proving Theorem \ref{thm3},  we have  studied the initial and boundary value problem of the Stokes equations in $\R_+\times (0,T)$ as follows:
\begin{align}\label{maineq-stokes}
\begin{array}{l}\vspace{2mm}
u_t - \De u + \na p =f, \qquad div \, u =0, \mbox{ in }
 \R_+\times (0,T),\\
\hspace{30mm}u|_{t=0}= h, \qquad  u|_{x_n =0} = g.
\end{array}
\end{align}


There are various literatures for the solvability of the  Stokes equations \eqref{maineq-stokes}
with homogeneous or nonhomogeneous boundary data. See \cite{cannone,giga,giga1,kato,koch,KS2,kozono1,sol1,Sol-2}, and references therein for the Stokes problem with homogeneous boundary data. See \cite{grubb1,  grubb3,grubb,KS1,KS2,raymond1,sol1,Sol-1}, and references therein for the Stokes problem with nonhomogeneous boundary data.

In \cite{ grubb1,  grubb3,grubb,KS1,KS2,sol1,Sol-1}, a boundary data in anisotropic space $B^{\al-\frac{1}{q},\frac{\al}{2}-\frac{1}{2q}}_{q0}(\partial\Omega\times {\mathbb R}_+)$, $\al>\frac{1}{q}$ has been considered.
 J.P. Raymond\cite{raymond1} showed the unique existence of weak solution $u\in B_2^{s,\frac{s}{2}}(\Omega \times (0,T))$, $0\leq s\leq 2$ in  a bounded domain when a nonzero initial data in $L^2(\Omega)$ and nonzero boundary data in $H^1(0,T;H^{-1}_2(S))$ are given. 
In \cite{farwig2}, R. Farwig, H. Kozono and H. Sohr also   showed the  existence of a very weak solution $u\in L^s(0,T;L^q(\Omega))$ (of Stokes equations) in an exterior domain when nonzero  initial in $B^{-\frac{2}{s}}_{q,s}$ and  nonzero boundary data  in $L^s(0,T;W^{-\frac{1}{q}}_q(\partial\Omega))$ for  $1<s<\infty, 3<q<\infty$ are given.

The following states  our result on  the unique solvability of the Stokes equations \eqref{maineq-stokes}.
\begin{theo}
\label{thm-stokes}
Let  $1<q<\infty$    Assume that    $h\in {B}_q^{-\frac{2}{q}}(\R_+)$ with $\mbox{div}h=0$, 
and   $
 g\in {B}_q^{-\frac{1}{q},-\frac{1}{2q}}(\Rn\times {\mathbb R}_+)$.
 In addition,  if $1<q\leq 3$, then we assume that $g-\Gamma*_{x}\tilde{h}\in B^{-\frac{1}{q}}_{q0}(\Rn \times {\mathbb R}_+)$ for some $\tilde{h}\in B^{-\frac{2}{q}}_q(\R)$ which is an solenoidal extension of $h$ to $\R$.
    Let $f=\mbox{div }{\mathcal F},\ {\mathcal F}\in L^p(\R\times{\mathbb R}_+)$ for some $p$  with
    $\al_1=1-(n+2)(\frac{1}{p}-\frac{1}{q})>0$. Then there is a unique weak solution $u\in  L^q({\mathbb R}^n_+\times (0,T) ) $ 
 with $\nabla u\in L^{p}_{loc}(\R_+\times (0,T))$
satisfying   the following inequality
\begin{align*}
\| u\|_{L^{q}({\mathbb R}^n_+\times (0,T))}
 &\leq c\max\{1,T^{\frac{1}{q}}\}\|h\|_{ {B}^{-\frac{2}{q}}_{q}({\mathbb
R}^{n}_+)}
+\max\{1,T^{\frac{1}{2q}}\}\|g\|_{   B^{-\frac{1}{q},-\frac{1}{2q}}_{q}({\mathbb
R}^{n-1} \times (0,T))}
\\
&\quad
   +\|g_n\|_{ L^q(0,T;\dot{B}^{-\frac{1}{q}}_q(\Rn))}
+cT^{\frac{\al_1}{2}}\|{\mathcal F}\|_{L^p(\R \times (0,T))}, \,\,q>3,
\end{align*}
and
\begin{align*}
\| u\|_{L^{q}({\mathbb R}^n_+\times (0,T))}
 &\leq c\max\{1,T^{\frac{1}{q}}\}\|h\|_{ {B}^{-\frac{2}{q}}_{q}({\mathbb
R}^{n}_+)}
+\max\{1,T^{\frac{1}{2q}}\}\|g-\Gamma*_{x}\tilde{h}\|_{   B^{-\frac{1}{q},-\frac{1}{2q}}_{q0}({\mathbb
R}^{n-1} \times (0,T))}
\\
&\quad
   +\|g_n\|_{ L^q(0,T;\dot{B}^{-\frac{1}{q}}_q(\Rn))}
+cT^{\frac{\al_1}{2}}\|{\mathcal F}\|_{L^p(\R \times (0,T))}, \,\, 1<q\leq 3.
\end{align*}



%
\end{theo}



We organize this paper as follows.
In section \ref{notation}, we introduce the  notations and the function spaces.
In section \ref{preliminary}  the  preliminary estimates in anisotropic spaces for the heat operator, Riesz operator, and Poisson operator are given.
In section \ref{zero}, we  consider Stokes equations \eqref{maineq-stokes} with  the zero force and the zero initial velocity, and give the proof of Theorem \ref{Rn-1}.
In section \ref{general}, we  complete the proof of Theorem \ref{thm-stokes} with the help of Theorem \ref{Rn-1} and the preliminary estimates in section \ref{preliminary}.
In section \ref{nonlinear}, we give the proof of Theorem \ref{thm3} applying the estimate of Theorem \ref{thm-stokes} to the  approximate solutions.

\section{Notations and Definitions}

\label{notation}
We denote by  $x'$ and $x=(x',x_n)$ the points of spaces $\Rn$ and $\R$, respectively.
The multiple derivatives are denoted by $ D^{k}_x D^{m}_t = \frac{\pa^{|k|}}{\pa x^{k}} \frac{\pa^{m} }{\pa t}$ for multi index
$ k$ and nonnegative integer $ m$.
For vector field $f=(f_1,\cdots, f_n)$ on $\R$, we write  $f'=(f_1,\cdots, f_{n-1})$ and $f=(f',f_n)$.
Throughout this paper we denote by $c$ various generic constants. 
Let ${\R}_+=\{x=(x',x_n): x_n>0\}$, $\overline{\R}_+=\{x=(x',x_n): x_n\geq 0\}$,
${\mathbb R}_+=(0,\infty)$.

For the Banach space $X$ and interval $I$, we denote by  $X'$ the dual space of $X$, and by $L^p(I;X), 1\leq p\leq \infty$  the usual Bochner space.
For $0< \theta<1$ and $1<p<\infty$,  denote by   $(X,Y)_{\theta,p}$ the real interpolation of the Banach space $X$ and $Y$.
For $1\leq p\leq \infty$, we write $p'=\frac{p}{p-1}$.

 Let   $\Omega$ be a $m$-dimensional Lipschitz domain, $m\geq 1$. Denote by
 $C^\infty_0(\Omega)$ stands for the collection of all complex-valued infinitely differentiable  functions in ${\mathbb R}^m$ compactly supported in $\Omega$. 
 Let $1\leq p\leq \infty$ and $k$ be a nonnegative integer.
 The norms of usual  Lebesque space $L^p(\Omega)$, the usual  Soboelv space  $W^k_p(\Omega)$ (Slobodetskii space $W^s_p(\Omega)$ for  noninteger $s>0$) and the usual   homogeneous Sobolev  spaces $\dot{W}^k_p(\Omega)$  are written by $\|\cdot\|_{L^p(\Omega)}, \ \|\cdot\|_{\dot{W}^k_p(\Omega)}, \ \|\cdot\|_{{W}^k_p(\Omega)}$, respectively.
    Note that $W^0_p(\Omega)=\dot{W}^0_p(\Omega)=L^p(\Omega)$.
   For $s\in {\mathbb R}$ and $1\leq p,q\leq \infty$, 
 denote by  $B^s_{p,q}(\Omega)$ and $\dot{B}^s_{p,q}(\Omega),1\leq p,q\leq \infty$  the usual Besov spaces and the homogeneous Besov spaces, respectively. 
For the simplicity, set $B^s_{p}(\Omega)=B^s_{p,p}(\Omega)$ and $\dot B^s_{p}(\Omega)=\dot B^s_{p,p}(\Omega)$.

 Denote by $\mbox{\r{B}}^{s}_{p}(\Omega)$  the set of distributions  $f\in B^{s}_p({\mathbb R}^m)$ which is  supported in $ \Omega$ with  norm
 $\|f\|_{\mbox{\r{B}}^{s}_{p}(\Omega)}=\|f\|_{B^{s}_{p}({\mathbb R}^m)} <\infty $.
%
%
{
 It is known that  $B^{s}_{p}(\bar{\Omega})=B^{s}_{p}(\bar{\Omega})
=\mbox{\r{B}}^{s}_{p}(\Omega) 
\mbox{ if }0\leq  s<\frac{1}{p}$,  $\mbox{\r{B}}^s_{p}({\Omega})=(\mbox{{B}}^{-s}_{p}(\Omega))'$  if $s<0$(when  $\Omega$ is Lipschitz domain).
%
%
}

It is also known that  $B^s_p(\Omega)=L^p(\Omega)\cap \dot{B}^s_p(\Omega)$ for $s>0$; $B^s_p(\Omega)=L^p(\Omega)+\dot{B}^s_{p}(\Omega)$ for $s<0$;
 $B^{s}_p(\Omega)=(L^p(\Omega), W^{s_1}_p(\Omega))_{\frac{s}{s_1}, p}$ for $0<s<s_1$; 
 $B^{s}_p({\mathbb R}^m)=(B^{s_1}_p({\mathbb R}^m), B^{s_2}_p({\mathbb R}^m))_{\theta, p}$ 
 for $s=(1-\theta)s_1+\theta s_2$ for $0<\theta<1$.
  See \cite{BL,St,Tr,Triebel, Triebel2} for  more properties of the   Besov spaces.

Let $I$ be an interval of ${\mathbb R}$. For $ k\in {\mathbb N}\cup\{0\}$, denote by  $W^{2k,k}_q(\Omega \times I)$  and $\dot{W}^{2k,k}_q(\Omega \times I)$  the usual anisotropic Sobolev  space (Slobodetskii space  $W^{s,\frac{s}{2}}_q(\Omega \times I)$ for noninteger $s>0$) and  the usual homogeneous anisotropic Sobolev space, respectively. Note that ${W}^{0,0}_p(\Omega\times I)=\dot{W}^{0,0}_p(\Omega\times I)=L^p(\Omega\times I).$
%

Now, we introduce  anisotropic Besov space and its properties (see chapter 4 of \cite{Triebel}, chapter 5 of \cite{Triebel2}, and chapter 3 of \cite{amman-anisotropic} for the definition of anisotropic spaces and their properties although different notations were used in each books).

Define anisotropic Besov space $B^{s,\frac{s}{2}}_p({\mathbb R}^m\times {\mathbb R})$ by
$$B^{s,\frac{s}{2}}_p({\mathbb R}^m\times {\mathbb R})=\left\{\begin{array}{l}
L^p({\mathbb R};B^s_p({\mathbb R}^m)\cap L^p({\mathbb R}^m;B^{\frac{s}{2}}_p({\mathbb R})\mbox{ if }s>0,\\
L^p({\mathbb R};B^s_p({\mathbb R}^m)+ L^p({\mathbb R}^m;B^{\frac{s}{2}}_p({\mathbb R})\mbox{ if }s<0,\\
(B^{-1,-\frac{1}{2}}_q({\mathbb R}^m\times {\mathbb R}),B^{1,\frac{1}{2}}_q({\mathbb R}^m\times {\mathbb R}))_{1,q}\mbox{ if }s=0.\end{array}\right.$$
The homogeneous anisotropic Besov space $\dot B^{s,\frac{s}{2}}_p({\mathbb R}^m\times {\mathbb R})$ is defined analogously.
The above  definition is equivalent to the definitions in \cite{amman-anisotropic,Triebel}.

Let $\Omega$ be a Lipshcitz domain of $\R$ and $I$ be an interval in ${\mathbb R}$. Let 
${\mathcal D}'$ be the distributions on  $\Omega \times I$. 
For $s\in {\mathbb R}$, an anisotropic  Besov space $B^{s, \frac s2 }_{q} (\Omega\times I)$ is defined by
\[
B^{s, \frac12 s}_q(\Omega\times I):=\{ f\in {\mathcal D}' \, | \, f=F|_{\Omega\times I}\mbox{ for some }F\in    B^{s, \frac12 s}_q(\R \times {\mathbb R})   \}
\]
with norm
$
\|f\|_{B^{s, \frac12 s}_{q} (\Omega\times I)}=\inf\{ \|F\|_{B^{\al, \frac12 \al}_q(\R \times {\mathbb R})}: F\in B^{s, \frac12 s}_q(\R \times {\mathbb R})\mbox{ with } F|_{\Omega\times I}=f\}.$
The homogeneous anisotropic spaces $
\dot{B}^{s, \frac s2 }_{q} (\Omega\times I)$ is   defined analogously.

Denote by $B^{s,\frac{s}{2}}_{p0}(\Omega \times (0,T))$  the set of distributions  $f\in B^{s,\frac{s}{2}}_p(\Omega\times (-\infty,T))$ which is  supported in $ \Omega \times (0,T)$ with
 $\|f\|_{B^{s,\frac{s}{2}}_{p0}(\Omega\times (0,T))} =\|f\|_{B^{s,\frac{s}{2}}_{p}(\Omega\times (-\infty,T))}<\infty  $.
%
%
 It is known that  $B^{s,\frac{s}{2}}_{p}(\Omega\times [0,T])=B^{s,\frac{s}{2}}_{p}(\Omega\times [0,T])=B^{s,\frac{s}{2}}_{p}(\Omega\times [0,T))
 =B^{s,\frac{s}{2}}_{p0}(\Omega\times (0, T))  \mbox{ if }0\leq s<\frac{2}{p}$ and $B^{s,\frac{s}{2}}_{p0}(\Omega\times (0, T))=(B^{-s,-\frac{s}{2}}_{p'}(\Omega\times (0, T)))'$ if $s<0$.


  The  properties of the  anisotropic Besov spaces are comparable with the properties of  Besov spaces:
$\dot B^{s,\frac{s}{2}}_{p}(\Omega\times I)   =L^p(I;\dot B^s_p(\Omega))\cap L^p(\Omega; \dot B^{\frac{s}{2}}_{p}(I))$ and $ B^{s,\frac{s}{2}}_{p}(\Omega\times I)   =L^p(\Omega\times I)\cap \dot B^{s,\frac{s}{2}}_{p}(\Omega\times I)$ for $s > 0$;
$\dot B^{s,\frac{s}{2}}_{p}(\Omega\times I)  =L^p(I;\dot B^s_p(\Omega))+ L^p(\Omega; \dot B^{\frac{s}{2}}_{p}(I))$
 and
 $ B^{s,\frac{s}{2}}_{p}(\Omega\times I)   =L^p(\Omega\times I)+ \dot B^{s,\frac{s}{2}}_{p}(\Omega\times I)$ for $s < 0$;
$B^{s,\frac{s}{2}}_{p}(\Omega\times I)   =(L^p(\Omega\times I),
    W^{2k,k}_{p}(\Omega\times I))_{\frac{s}{2k},p}$ for $ 0<s<2k$
;
    $B^{\al,\frac{\al}{2}}_{p}({\mathbb R}^m\times {\mathbb R})   =(B^{\al_1,\frac{\al_1}{2}}_{p}({\mathbb R}^m\times {\mathbb R}),
   B^{\al_2,\frac{\al_2}{2}}_{p}({\mathbb R}^m\times {\mathbb R}))_{\te,p}, \, 0<\te<1, \,\, \al = (1-\te)\al_1 + \te \al_2$ for any real number $\al_1< \al_2$. See \cite{amman-anisotropic,Triebel,  Triebel2} for  more properties of the   anisotropic Besov spaces.

\begin{defin}[Weak solution to the Stokes equations]
\label{stokesdefinition}
Let  $1<q<\infty$. Let $h,g, f=\mbox{div}{\mathcal F}$ satisfy the same hypothesis as in Theorem \ref{thm-stokes}.
Then a vector field $u\in L^q(\R_+\times (0,T))$ with $\nabla u\in L^p_{loc}(\R_+\times (0,T))$ is called a weak solution of the Stokes system \eqref{maineq-stokes} if the following conditions are satisfied:\\
 $\bullet$(In case $\infty>q>3$)
 \begin{align*}
-\int^T_0\int_{\R_+}u\cdot \Delta \Phi dxdt
=\int^T_0\int_{\R_+}u\cdot \Phi_t-{\mathcal F}:\nabla \Phi dxdt
+<h,\Phi(\cdot,0)>_{\R_+}
-<g,\frac{\partial \Phi}{\partial x_n}>_{\Rn\times {\mathbb R}_+}
\end{align*}
for each $\Phi\in C^\infty_0(\overline{\R}_+\times [0,T))$ with $\mbox{div}_x\Phi=0$, $\Phi|_{x_n=0}=0$, where
 $<\cdot, \cdot>_{\R_+}$ denotes the duality paring between between $B^{-\frac{2}{q}}_q(\R_+)$ and $B^{\frac{2}{q}}_{q'}(\R_+)$ and  $<\cdot, \cdot>_{\Rn\times {\mathbb R}_+}$ denotes the duality paring between between $B^{-\frac{1}{q},-\frac{1}{2q}}_{q}(\Rn\times (0,T))$ and $B^{\frac{1}{q},\frac{1}{2q}}_{q'}(\Rn\times (0,T))$.
\end{defin}

$\bullet$(In case $1<q\leq 3$)
 \begin{align*}
-\int^T_0\int_{\R_+}(u-v)\cdot \Delta \Phi dxdt=\int^T_0\int_{\R_+}(u-v)\cdot \Phi_t-{\mathcal F}:\nabla \Phi dxdt
-<g-v|_{x_n=0},\frac{\partial \Phi}{\partial x_n}>_{\Rn\times {\mathbb R}_+)}
\end{align*}
for each $\Phi\in C^\infty_0(\overline{\R}_+\times [0,T))$ with $\mbox{div}_x\Phi=0$, $\Phi|_{x_n=0}=0$, where $v=\Gamma_t*_x\tilde{h}$ and
   $<\cdot, \cdot>_{\Rn\times {\mathbb R}_+}$ denotes the duality paring between between $B^{-\frac{1}{q},-\frac{1}{2q}}_{q0}(\Rn\times (0,T))$ and $B^{\frac{1}{q},\frac{1}{2q}}_{q'}(\Rn\times (0,T))$.

\begin{defin}[Weak solution to the Navier-Stokes equations] Let  $\infty>q>n+2$.
Let $h,g$ satisfy the same hypothesis as in Theorem \ref{thm3}.
Then a vector field $u\in  L^q(\R_+\times (0,T))$  with $\nabla u\in L^p_{loc}(\R_+\times (0,T))$ for some $1<p\leq q$ is called a weak solution of the Navier-Stokes system \eqref{maineq2} if the following conditions are satisfied:\\
\begin{align*}
-\int^T_0\int_{\R_+}u\cdot \Delta \Phi dxdt=\int^T_0\int_{\R_+}u\cdot (\Phi_t+(u\otimes u):\nabla \Phi) dxdt
+<h,\Phi(\cdot,0)>_{\R_+}
-<g,\frac{\partial \Phi}{\partial x_n}>_{\Rn\times {\mathbb R}_+}
\end{align*}
for each $\Phi\in C^\infty_0(\overline{\R}_+\times [0,T))$ with $\mbox{div}_x\Phi=0$, $\Phi|_{x_n=0}=0$.

\end{defin}

\section{\bf Preliminaries.}

\label{preliminary}
\setcounter{equation}{0}

\subsection{Basic Theories}

%
%

According to the usual trace theorem, if  $u\in B^s_p(\R_+)$ ($u\in W^s_p(\R_+)$), then $u|_{x_n=0}\in B_p^{s-\frac{1}{p}}(\Rn)$  for $s>\frac{1}{p}$ (see \cite{BL}), and  if $u\in B^{s,\frac{s}{2}}_p(\R_+\times (0,T))$ ($u\in W^{s,\frac{s}{2}}_p(\R_+\times (0,T))$), then $u|_{x_n=0}\in B^{s-\frac{1}{p},\frac{s}{2}-\frac{1}{2p}}_p(\Rn\times (0,T))$  for $s>\frac{1}{p}$ and $u|_{t=0}\in B^{s-\frac{2}{p}}_p(\R_+)$  for $s>\frac{2}{p}$
(see \cite{amman-anisotropic,Triebel,Triebel2}).
On the other hand, for a solenoidal vector field  $u\in L^p(\R_+)$, $1<p<\infty$ it holds that
$u_n \in \dot B^{-\frac1p}_p (\Rn)$ with
\begin{align}
\label{tracetheorem}
\| u_n\|_{\dot B^{-\frac1p}_p(\Rn)} \leq c \| u\|_{L^p(\R_+)}.
\end{align}



Let $R=(R_1,\cdots, R_{n})$ be the Riesz operator on $\R$.
It is the well known fact that $R_i$ is bounded operator from $B^s_p(\R)$ to $B^s_p(\R)$ (from $W^k_p(\R)$ to $W^k_p(\R)$, $k=0,\pm 1,\cdots$) for $s\in {\mathbb R}$ and $1<p<\infty$ 
(see \cite{St} for the reference).
Using the fact that $B^{s,\frac{s}{2}}_{q}(\Omega\times (0,T))=L^p(0,T;B^s(\Omega))\cap L^p(\Omega;B^{\frac{s}{2}}_p(0,T))$ 
for $s>0$ and $R_i$ is self-adjoint operator,
the following boundedness property  holds for anisotropic Besov spaces as follows:
\begin{align}
\label{propriesz2}
\|Rf\|_{ B^{s,\frac{s}{2}}_{q}({\mathbb R}^n\times (0,T))}&\leq c\|f\|_{ B^{s,\frac{s}{2}}_{q}({\mathbb R}^n\times (0,T))},s\in {\mathbb R}, 1<q<\infty,
\\
(\|Rf\|_{ W^{k,\frac{k}{2}}_{q}({\mathbb R}^n\times (0,T))}&\leq c\|f\|_{ W^{k,\frac{k}{2}}_{q}({\mathbb R}^n\times (0,T))},k=0,1,2,\cdots).
\end{align}

\subsection{Estimate of the heat operator}

Define three types of heat operator $ T_1, T_2,T_1^*, T_2^*$ by
\begin{align*}
 T_1f&=\int^t_{-\infty}\int_{\R}\Gamma(x-y,t-s)f(y,s)dyds, \\
T_2g&=\int^t_{-\infty}\int_{\Rn}\Gamma(x'-y',x_n,t-s)g(y',s)dy'ds,\\
T_1^*f(y,s) & = \int^{\infty}_s \int_{\R} \Ga(x-y, t-s) f(x,t) dxdt,\\
T_2^*g(y,t) &=\int^\infty_s \int_{\Rn} \Ga(x'-y', y_n, t-\tau)
g(x', \tau) dx' dt.
\end{align*}

Observing that $T^*_1$ is the adjoint operator of $T_1$, we can derive the following estimate for $T_1$ and $T_1^*$.
 \begin{lemm}\label{lemma0115}
 Let $1<p<\infty$ and $0\leq \al\leq 2$.
\begin{align*}
%
%
%
%
 \| T_1f\|_{\dot{W}^{2k,k}_p (\R \times {\mathbb R})} + \| T^*_1 f\|_{\dot{W}^{2k,k}_p (\R \times {\mathbb R})}
&\leq c \| f\|_{\dot B^{2k-2,k -1}_p (\R \times {\mathbb R})}, \quad k=0,1,\\
\| T_1f\|_{\dot B^{\al, \frac{\al}{2}}_p (\R \times {\mathbb R})} + \| T^*_1 f\|_{\dot B^{\al, \frac{\al}{2}}_p (\R \times {\mathbb R})}
&\leq c \| f\|_{\dot B^{\al-2, \frac{\al}{2}-1}_p (\R \times {\mathbb R})},  \quad 0<\al< 2.
\end{align*}
\end{lemm}

%
%
%
%
%
%
%
%

Using the result of Lemma \ref{lemma0115}, the following estimate  for $T_2$ and $T_2^*$ can be derived. 
\begin{lemm}
\label{lem-T}
 Let $1<p<\infty$ and $0\leq \al\leq 2$.
Then
\begin{align*}
\| T_2g \|_{\dot{W}^{2k,k}_p(\R_+ \times {\mathbb R})}+\| T_2^*g \|_{\dot{W}^{2k,k}_p(\R_+ \times {\mathbb R})} & \leq c \|
g\|_{\dot{B}^{2k-1-\frac{1}{q},k-\frac{1}{2q}}_q(\Rn \times {\mathbb R})}, \quad k =0, \, 1,\\                                                                                                        
\| T_2g \|_{\dot{B}^{\al,\frac{\al}{2}}_q(\R_+ \times {\mathbb R})}+\| T_2^*g \|_{\dot{B}^{\al,\frac{\al}{2}}_q(\R_+ \times {\mathbb R})} & \leq c \| g\|_{\dot{B}^{\al-1-\frac{1}{q},\frac{\al}{2}-\frac{1}{2q}}_q(\Rn \times {\mathbb R})},\quad 0<\al< 2.
\end{align*}
\end{lemm}
The above lemma  will be useful for the proof of Theorem \ref{Rn-1} and also for the estimate of $\Gamma_t*h|_{x_n=0}$.

\begin{lemm}
\label{proheat1}
Let $1<q<\infty$. Define the heat operator $T_0$ by $T_0h=\int_{\R}\Gamma(x-y,t)h(y)dy$.
Then,
\begin{align}
\|T_0h\|_{L^q(\R\times (0,T))}&\leq
c\max\{1,T^{\frac{1}{q}}\}\|h\|_{{B}^{-\frac{2}{q}}_q(\R)},\\
\|T_0h|_{x_n=0}\|_{B^{-\frac{1}{q},-\frac{1}{2q}}_{q}(\Rn\times (0,T))}&\leq c
\max\{1, T^{\frac1{2q}}\}\|h\|_{B^{-\frac{2}{q}}_q(\R)}.
\end{align}

\end{lemm}


\begin{lemm}
\label{propheat2}
Let    $1<p, q<\infty$ with $1-(n+2)(\frac{1}{p}-\frac{1}{q})>0$.
For $f\in 
L^p(\R\times (0,T))
$, define $u$ by  $u(x,t)=
\int^t_0\int_{\R}D_x \Ga(x-y,t-s)f(y,s)dyds.$ Let $\al_1=1-(n+2)(\frac{1}{p}-\frac{1}{q}).$
Then $u\in {W}^{1,\frac{1}{2}}_{p0}(\R\times (0,T))$ with
\begin{align*}
\|u\|_{{W}^{1,\frac{1}{2}}_{p0}(\R\times (0,T))}
&\leq c\max\{1,T^{\frac{\al_1}{2}}\}
\|f\|_{L^p(\R\times (0, T))},\\
\|u\|_{L^q(\R\times (0,T)}
&\leq 
cT^{\frac{\al_1}{2}}\|f\|_{L^p(\R\times (0, T))}.
\end{align*}
Moreover, $u|_{x_n=0} \in B^{-\frac1q, -\frac1{2q} }_{q0}(\Rn\times (0,T)) $ with  
%
%
%
%
\begin{align*}
\|u|_{x_n=0}\|_{  B^{-\frac1q, -\frac1{2q} }_{q0}(\Rn\times (0,T))} &\leq cT^{\frac{\al_1}2}
\|f\|_{L^p(\R\times (0, T))}.
\end{align*}

\end{lemm}

\section{\bf Stokes equations \eqref{maineq-stokes} with $f=0$ and $h=0$}

\label{zero}

\setcounter{equation}{0}

Let $(w,r)$ be the solution of the equations
 \begin{align}\label{maineq1}
\begin{array}{l}\vspace{2mm}
w_t - \De w + \na r =0, \qquad div \, w=0, \mbox{ in }
 \R_+\times (0,T),\\
\hspace{30mm}w|_{t=0}= 0, \qquad  w|_{x_n =0} = G.
\end{array}
\end{align}
Let
\begin{align*}
K_{ij}(x,t) & = -2 \delta_{ij}D_{x_n}  \Ga(x,t)  +4 D_{x_j}\int_0^{x_n} \int_{\Rn}  D_{z_n}  \Ga(z,t)  D_{x_i} N(x-z)  dz,
\end{align*}
where $  \Ga$ and $N$ are the fundamental solutions of the heat equation and the Laplace equation in $\R$, respectively, that is,
\[
 \Gamma(x,t)=\left\{\begin{array}{ll} \vspace{2mm}
 \frac{c}{ (2\pi t)^{\frac{n}{2}}}e^{-\frac{|x|^2}{4t}}&\mbox{ if }t>0,\\
 0& \mbox{ if }t\leq 0,
 \end{array}\right. \quad \mbox{and} \quad    N(x) = \left\{\begin{array}{ll}
 \vspace{2mm}
  \frac{1}{\om_n (2-n)|x|^{n-2}}&\mbox{ if }n\geq 3,\\
 \frac{1}{2\pi}\ln |x|&\mbox{ if }n=2.\end{array}\right.
 \]
In \cite{sol1}, an explicit formula for $w$
 of the Stokes equations \eqref{maineq1} with boundary data $G=(G', 0)$  is
 obtained by
\begin{align}\label{simple}
w_i(x,t)& = \sum_{j=1}^{n-1}< K_{ij}( x'-\cdot,x_n,t-\cdot),G_j>_{\Rn\times {\mathbb R}_+}.
\end{align}
Here $<\cdot,\cdot>$ is a duality paring between  $B^{-\al+\frac{1}{q},-\frac{\al}{2}+\frac{1}{2q}}_q(\Rn\times (0,T))$ and $B^{\al-\frac{1}{q},\frac{\al}{2}-\frac{1}{2q}}_{q0}(\Rn\times {\mathbb R}_+)$ (if  $G$ is a function, then  $< K_{ij}( x'-\cdot,x_n,t-\cdot),G_j>_{\Rn\times {\mathbb R}_+}=\int_0^t \int_{\Rn} K_{ij}( x'-y',x_n,t-s)G_j(y',s) dy'ds).$
 \begin{theo}\label{Rn-1}
Let $1< q<\infty$ and $0<T<\infty$.
Let  $G \in B^{-\frac{1}{q},-\frac{1}{2q} }_{q0} ({\mathbb R}^{n-1} \times {\mathbb R}_+)
$ with
$G_n=0$. 
 Let $w$ be the vector field defined by \eqref{simple}. 
%
Then $w\in L^q(\R_+\times (0,T))$ with
 \begin{align*}
\| w\|_{ L^q({\mathbb R}^n_+\times (0,T))}
 \leq c\max\{1,T^{\frac{1}{2q}}\}\|G\|_{B^{-\frac{1}{q},-\frac{1}{2q}}_{q0}({\mathbb    R}^{n-1} \times (0,T))}.
 \end{align*}

\end{theo}

\begin{proof}


%
%

By the definition of the space $ B^{-\frac1q,\frac{1}{2q}}_{q0}(\Rn\times (0,T))$,
the  zero extension of $G$ is in $ B^{-\frac1q,\frac{1}{2q}}_{q}(\Rn\times (-\infty,T))$. Again, by the  definition of the space $ B^{-\frac1q,\frac{1}{2q}}_{q0}(\Rn\times (-\infty,T))$, there is $\tilde{G}\in  B^{-\frac1q,\frac{1}{2q}}_{q}(\Rn\times {\mathbb R})$ with $
\tilde{G}|_{\Rn\times (0,T)}=G,$ $\mbox{supp}\tilde{G}\subset \Rn\times (0,\infty)$ and  $\|\tilde{G}\|_{ B^{-\frac1q,\frac{1}{2q}}_{q}(\Rn\times {\mathbb R})}
\leq c\|G\|_{B^{-\frac1q,\frac{1}{2q}}_{q0}(\Rn\times (0,T))}.$
Hence, without loss of generality, we assume  that $G\in B^{-\frac{1}{q},-\frac{1}{2q}}_{q}(\Rn\times {\mathbb R})$,  $\mbox{supp  }G\subset \Rn\times {\mathbb R}_+$  with $\|{G}\|_{ B^{-\frac1q,\frac{1}{2q}}_{q}(\Rn\times {\mathbb R})}
\leq c\|G\|_{B^{-\frac1q,\frac{1}{2q}}_{q0}(\Rn\times (0,T))}$ (By density argument, we may assume that $G \in C_0^\infty(\Rn \times {\mathbb R})$).

According to \cite{CC}, $w$ can be rewritten by the following form
\begin{equation}\label{C60K-april7}
w_i(x,t) = -{\mathcal T} G_i(x,t)   -4\delta_{in} {\mathcal I} \Big(\sum_{j=1}^{n-1} \frac{\pa}{\partial x_j}
        {\mathcal T}G_j\Big)(x,t)+ 4
    \frac{\partial}{\partial x_i} {\mathcal S}\Big(\sum_{j=1}^{n-1}\frac{\pa}{\partial x_j} {\mathcal T}G_j\Big)(x,t), \,\, i=1,\cdots,n,
\end{equation}
where ${\mathcal T}$, ${\mathcal S}$  and ${\mathcal I}$ are defined by
\begin{align}\label{C70K-april7}
{\mathcal T}G_i(x,t)&=\int_{-\infty}^t \int_{\Rn} D_{x_n}\Ga(x'-y', x_n, t-\tau)
G_i(y', \tau) dy' d\tau,\\
\label{1027} {\mathcal I}f (x,t) &= \int_{{\mathbb R}^{n-1}} N(x' -y', 0) f(y',x_n,t)
dy',\\
\label{C80K-april7}
{\mathcal S}f (x,t) &= \int_0^{x_n} \int_{{\mathbb R}^{n-1}} N(x-y) f (y,t) dy.
\end{align}
%


Observe that ${\mathcal T}=D_{x_n}T_2$, where  $T_2$ is the heat operator defined in section \ref{preliminary}.
From Lemma \ref{lem-T},  we have
\begin{align}
\label{v-42}
\| {\mathcal T}   G \|_{L^{q}(\R_+ \times {\mathbb R})}  \leq
c  \|   G \|_{ \dot B^{-\frac1q,-\frac1{2q}}_q(\Rn \times {\mathbb R})}.
\end{align}

Direct computation shows that  for $1 \leq j \leq n-1$ 
\begin{equation}\label{C100K-april7-1}
 {\mathcal I}\Big(\sum_{j=1}^{n-1} \frac{\pa}{\partial x_j}{\mathcal T}G_j\Big) = \sum_{j=1}^{n-1}R^{'}_j {\mathcal T}G_j,
\end{equation}
where   $R'=(R_1',\cdots,R_{n-1}')$ is $n-1$ dimensional Riesz operator.
By the well known property of Riesz operator
 we have
\begin{align}
\label{v-31}
 \|  {\mathcal I}\Big(\frac{\pa}{\partial x_j}{\mathcal T}G_j\Big) \|_{L^q ({\mathbb R}_+^n\times {\mathbb R})} dt & \leq c
\sum_{j=1}^{n-1}  \| {\mathcal T}G_j
\|_{L^q ({\mathbb R}^{n-1}\times {\mathbb R})}\leq  c  \|G\|_{\dot B^{ -\frac1q,  - \frac1{2q}}_q(\Rn \times {\mathbb R})}.
\end{align}

Let $
f(x,t)=\sum_{j=1}^{n-1} \frac{\pa}{\partial x_j}{\mathcal T}G_j(x,t).$
Direct computation also shows that
${\mathcal S}f$ solves
\begin{equation}\label{CK100-april7}
\De {\mathcal S}f(x,t)=
\mbox{div}F
      \mbox{ in }\,\,{\mathbb R}^n_+\mbox{ for each }t>0,\quad {\mathcal S}f|_{x_n=0}=0,
\end{equation}
where
\begin{equation}\label{C100K-april7}
F_j: = -\frac12{\mathcal T}G_j,\qquad j=1\cdots, n-1,
\quad
F_n: = {\mathcal I}\sum_{j=1}^{n-1} \frac{\pa}{\partial x_j}{\mathcal T}G_j(x,t).
\end{equation}
%
 By the solution representation of
Laplace equation \eqref{CK100-april7},  ${\mathcal S}f$ can be rewritten by  the formula
\begin{align*}
{\mathcal S}f(x,t) & = -\int_{{\mathbb R}^n_+} (N(x-y) - E(x-y^*) ) div \, F(y,t) dy\\
        & = \int_{{\mathbb R}^n_+} \nabla_x (N(x-y) - E(x-y^*) ) \cdot F(y,t) dy.
\end{align*}
%
Using Calderon Zygmund inequality, we have 
\begin{equation}
\label{CK25-april11-1}
 \| D_x{\mathcal S}f\|_{L^q({\mathbb R}^n_+\times {\mathbb R})}\leq c
  \|F\|_{L^q({\mathbb R}^n_+\times {\mathbb R})}\leq c \|G\|_{\dot B^{ -\frac1q,  - \frac1{2q}}_q(\Rn \times {\mathbb R})}.
\end{equation}

Combining  \eqref{v-42}, \eqref{v-31} and \eqref{CK25-april11-1},  we have
\begin{equation}\label{C110K-april7-1}
\| w\|_{L^q(\R_+ \times {\mathbb R})} \leq c \|G\|_{\dot B^{ -\frac1q,  - \frac1{2q}}_q(\Rn \times {\mathbb R})}.
\end{equation}
On the other hand, by Young's theorem and Minkovski's theorem, we have
 \begin{equation}\label{y11}
\| w\|_{L^q(\R_+ \times (0, T))} \leq cT^{\frac{1}{2q}}
\|  G\|_{L^{q}(\Rn \times (0,T)) }.
\end{equation}
 Recall that ${B}^{s,\frac{s}{2}}_{q0}(\Omega\times (0,T))=\dot{B}^{s,\frac{s}{2}}_{q0}(\Omega\times (0,T))+L^q(\Omega\times (0,T))$ for $s<0$. Combining
from  \eqref{C110K-april7-1} and \eqref{y11}, we have
 \begin{align}\label{y22}
\| w\|_{L^q(\R_+ \times (0, T))} &\leq c\max\{1,T^{\frac{1}{2q}}\}
 \|  G\|_{B^{ -\frac1q,  -\frac1{2q}}_{q0}(\Rn \times (0,T)) }.
\end{align}


\end{proof}

\begin{rem}
\label{regularity-zero}
Let $G_{ij}^*(x,y,t)= D_{x_j}\int_0^{x_n} \int_{\Rn}   \Ga(z-y^*,t)  D_{x_i} N(x-z)  dz$, $y^*=(y',-y_n)$.
It is known that \begin{align}
\label{gauss}
|D^{s}_{t} D^{k}_{x} D_{y}^{m} G^*_{ij}(x,y,t)|
& \leq  \frac{c}{t^{s + \frac{m_n}2} (|x-y^*|^2 +t )^{\frac{n+k'+m'}2 } (x_n^2 +t)^{\frac{k_n}2 }}e^{-\frac{cy_n^2}{t}},
\end{align}
where $ 1 \leq  i \leq n$ and $1 \leq j \leq n-1$ (see Proposition 2.5 of \cite{Sol-2}).
Using the properties of
heat kernel $\Gamma_t$ and the estimates of
$G^*_{ij}$, we have
\begin{align*}
|D^{s}_{t} D^{k}_{x} D_{x_j}K_{ij}(x'-y',x_n,t)| \leq  \frac{c}{t^{s + \frac{1}2} (|x'-y'|^2+x_n^2 +t )^{\frac{n+k'}2 } (x_n^2 +t)^{\frac{k_n}2 }}.
\end{align*}
Using this estimate of $K_{ij}$, direct computation shows that
\[
\|D_xw(\cdot, x_n,t)\|_{L^q(\Rn)}\leq ct^{\frac{1}{2}}x_n^{-2}\|G\|_{L^q(\Rn \times (0,T))}.
\]
and
\[
\|D_xw(\cdot, x_n,t)\|_{L^q(\Rn \times (0, T))}\leq ct^{\frac{1}{2}}x_n^{-2-\frac{1}{q}}\|G\|_{\dot{B}^{-\frac{1}{q},-\frac{1}{2q}}_{q0}(\Rn \times (0,T))}.
\]
\end{rem}

\section{\bf
Proof of Theorem \ref{thm-stokes}
}
\label{general}
\setcounter{equation}{0}

Let us   consider the Stokes equations \eqref{maineq-stokes}  with general nonhomogeneous data $h, f, g$ with
 $f=\mbox{\rm div}{\mathcal F}.$
Below, we give a solution formula of the  Stokes equations \eqref{maineq-stokes} decomposed by  four vector field, $v, V, \nabla \phi$ and $w$ which will be defined in section \ref{decomposition}.

\subsection{Solution formula}
\label{decomposition}
Let $\tilde{\mathcal F}$ be an extension of ${\mathcal F}$ to $\R\times {\mathbb R}_+$, and let $\tilde{f}=\mbox{div}\tilde{\mathcal F}$.
%
Define  the projection operator ${\mathbb P}$ by
\[
[{\mathbb P}\tilde{f}]_j(x,t)=\delta_{ij}\tilde{f}_i+D_{x_i}D_{x_j}\int_{\R}N(x-y)\tilde{f}_i(y,t)dy=\delta_{ij}\tilde{f}_i+R_iR_j\tilde{f}_i,\]
and define ${\mathbb Q}$ by
\[
{\mathbb Q}\tilde{f}=-D_{x_i}\int_{\R}N(x-y)\tilde{f}_i(y,t)dy.\]
  Then
 \[
    \mbox{div }{\mathbb P}\tilde{f}=0\mbox{ in }\R\times (0,T)\mbox{ and }
  \tilde{f}={\mathbb P}\tilde{f} +\nabla {\mathbb Q}\tilde{f}.\]
Define  $V$ by
 \begin{align}
  V(x,t)=\int^t_{0}\int_{\R}\Gamma(x-y,t-s){\mathbb P}\tilde{f}(y,s)dyds.
 \end{align}
   Observe that
  $V$  satisfies the equations
 \begin{align}
 \label{heatequations1}
V_t - \De V  ={\mathbb P}\tilde{f},\ \mbox{div }V=0 \mbox{ in }
 \R\times (0,T),  \quad
V|_{t=0}= 0\mbox{ on }\R.
\end{align}
Observe that  $({\mathbb P}\tilde{f})_j=D_{x_k}\Big(\delta_{ij}\tilde{F}_{ki}+R_iR_j\tilde{F}_{ki}\Big)$ for $\tilde{f}=\mbox{div }\tilde{\mathcal F}$. 
Hence $V$ can be rewritten by
\begin{align}
\label{f1}
  V_j(x,t)=-\int^t_{0}\int_{\R}D_{y_k}\Gamma(x-y,t-s)\Big(\delta_{ij}\tilde{F}_{ki}+R_iR_j\tilde{F}_{ki}\Big)(y,s)dyds.
 \end{align}

Let
 $\widetilde{h}$ be an  extension of  $h$ satisfying that
 \[
 \mbox{div }\widetilde{h}=0\mbox{ in }\R.
 \]
Define $v$ by
  \begin{equation}
  \label{f2}
  v(x,t)=\int_{\R}\Gamma(x-y,t) \widetilde{h}(y)dy.
  \end{equation}
Observe that $v$ satisfies the equations
 \begin{align}
 \label{heatequations2}
v_t - \De v  =0,\  \mbox{div}v=0 \mbox{ in }
 \R\times (0,T),\quad
v|_{t=0}= \tilde{h} \mbox{ on }\R.
\end{align}


  Define $\phi$ by
 \begin{equation}
 \label{f3}
 \phi(x,t)=2\int_{\Rn}N(x'-y',x_n)\Big(g_n(y',t)-v_n(x',0,t)-V_n(x',0,t)\Big)dy'.
 \end{equation}
 Observe that
 \begin{equation}
 \label{laplace}
  \Delta \phi=0,  \nabla \phi|_{x_n=0}=\Big(R'(g_n-v_n|_{x_n=0}-V_n|_{x_n=0}), g_n-v_n|_{x_n=0}-V_n|_{x_n=0}\Big).
 \end{equation}
Note that $ \nabla\phi|_{t=0}=0$ if $g_n|_{t=0}=h_n|_{x_n=0}$.

Let  $G=(G',0)$, where
\begin{equation}
\label{G}
G'=(G_1,\cdots, G_{n-1})=g'-v'|_{x_n=0}-V'|_{x_n=0}-R'(g_n-v_n|_{x_n=0}-V_n|_{x_n=0}).
\end{equation}
Note that $G'|_{t=0}=0$ if $g|_{t=0}=h|_{x_n=0}$.
 Let $w$ be the vector field defined by the formula \eqref{simple} with boundary data $G=(G',0)$ for $G'$  as defined in \eqref{G}.
 %
%
 Then,
 \begin{equation}
 u=w+\nabla \phi+v+V\mbox{ and } p=r-\phi_t+{\mathbb Q}\tilde{f}
 \end{equation}
  satisfies formally  the nonstationary following Stokes equations  \eqref{maineq-stokes}.%


\subsection{
Estimates of $u=v+V+\nabla \phi+w$}

\label{proof.stokes}

Choose $\tilde{h}\in B^{-\frac{2}{q}}_q(\R)$
  so that $\tilde{h}|_{\R_+}=h$ and $\mbox{\rm div }\tilde{h}=0$. Choose    $\tilde{\mathcal F} \in L^p( \R\times {\mathbb R}_+)$
   so that $\tilde{\mathcal F}|_{\R_+\times {\mathbb R}_+}={\mathcal F}$.  Let $\tilde{f}=\mbox{div }\tilde{\mathcal F}.$

Let   $V$, $v$ and $ \phi$ be the corresponding vector fields defined by \eqref{f1},  \eqref{f2}, and \eqref{f3}, respectively, and  
%
let $w$ be  defined by \eqref{simple} with $G$  as defined by    \eqref{G}.

%


1)
From Lemma \ref{proheat1}, we have
with
\begin{align*}
\|v\|_{L^q(\R_+\times (0,T))}&\leq c \max\{1,T^{\frac{1}{q}}\} \|h\|_{B^{-\frac{2}{q}}_q(\R_+)}.
\end{align*}

2) 
 By the $L^p$ boundedness of the Riesz operator (see \eqref{propriesz2}) we have
$\|R_iR_j \tilde{F}_{kl}\|_{L^p(\R\times {\mathbb R}_+)}\leq c\|\mathcal F\|_{L^p(\R_+\times {\mathbb R}_+)}$. Hence, from Lemma \ref{propheat2} we have
\begin{align*}
\|V\|_{L^q(\R\times (0,T))}                                                       
&\leq  cT^{\frac{\al_1}{2}}\|{\mathcal F}\|_{L^p(\R \times (0,T))}.\end{align*}

3) Since $div\, V=0,\,\, div\, v =0$ in $\R_+ \times (0,T)$, $v_n$ and $V_n$ have trace (see  \eqref{tracetheorem}) with
\begin{align}
\label{t1}
\| V_n(t)|_{x_n =0} \|_{\dot{B}^{-\frac1q}_q(\Rn)}\leq c \|V(t)\|_{ L^q (\R_+)},\quad
  \| v_n(t)|_{x_n =0} \|_{ \dot{B}^{-\frac1q}_q(\Rn)} \leq c \|v(t)\|_{ L^q (\R_+)}.
\end{align}
This leads to the estimate
\begin{align}
\label{t4}
\|v_n|_{x_n=0}\|_{L^q(0,T;\dot{B}_q^{-\frac{1}{q}}(\Rn))}\leq \|v\|_{L^q(0,T;L^q(\R_+))}=\|v\|_{L^q(\R_+\times(0,T))},
\end{align}
\begin{align}
\label{t5}
\|V_n|_{x_n=0}\|_{L^q(0,T;\dot{B}_q^{-\frac{1}{q}}(\Rn))}\leq \|V\|_{L^q(0,T;L^q(\R_+))}=\|V\|_{L^q(\R_+\times(0,T))}.
\end{align}

4)
Let $P_{x_n}$ be the Poisson operator defined by
\[P_{x_n}f(x)=
c_n\int_{\Rn}\frac{x_n}{(|x'-y'|^2+x_n^2)^{\frac{n}{2}}}f(y')dy',\]
which satisfies the Laplace equation
\[
-\Delta P_{x_n}f=0\mbox{ in }\R_+,\ P_{x_n}f|_{x_n=0}=f.\]
Observe that
\begin{align*}
D_{x_n}\phi(x,t)&=2\int_{\Rn}D_{x_n}N(x'-y',x_n)(g_n(y',t)-v_n(y',0,t)-V_n(y',0,t))dy'\\
&=P_{x_n}(g_n-v_n|_{y_n=0}-V_n|_{y_n=0}),\\
D_{x'}\phi(x,t)&=2\int_{\Rn}D_{x_n}N(x'-y',x_n)R'(g_n-v_n|_{y_n=0}-V_n|_{y_n=0})(y',t)dy'\\
&=P_{x_n}R'(g_n-v_n|_{y_n=0}-V_n|_{y_n=0}).
\end{align*}
%
Since Poisson operator $P_{x_n}$ is bounded from $\dot{B}^{-\frac{1}{q}}_q$ to $L^q$ (see \cite{St} for the reference) and Riesz operator $R_i'$ is $L^q$ bounded, we have that 
\begin{align*}
\|\nabla \phi\|_{L^q(\R_+\times (0,T))}&\leq c\|g_n-v_n|_{x_n=0}-V_n|_{x_n=0}\|_{L^q(0,T;\dot{B}^{-\frac{1}{q}}_q(\Rn))}\\
&\quad+c\|R'\Big(g_n-v_n|_{x_n=0}-V_n|_{x_n=0}\Big)\|_{L^q(0,T;\dot{B}^{-\frac{1}{q}}_q(\Rn))}\\
&\leq c\|g_n-v_n|_{x_n=0}-V_n|_{x_n=0}\|_{L^q(0,T;\dot{B}^{-\frac{1}{q}}_q(\Rn))}\\
&\leq c \big( \|g_n\|_{L^q(0,T;\dot{B}^{-\frac{1}{q}}_q(\Rn))}
+\|v\|_{L^q(\R_+\times (0,T))}+\|V\|_{L^q(\R_+\times (0,T))} \big).
\end{align*}

5)
%
%
%
Applying the last estimates of Lemma \ref{proheat1}, 
$ v|_{x_n =0} \in B_{q}^{ -\frac1q,  -\frac1{2q}}(\Rn \times (0,T))$ with
\begin{align*}
\| v|_{x_n =0} \|_{B_{q}^{ -\frac1q,  -\frac1{2q}}(\Rn \times (0,T))}\leq c
\max\{1,T^{\frac{1}{2q}}\}\|h\|_{ {B}_q^{ -\frac2q}(\R_+)}.
\end{align*}
Applying the last  estimates of Lemma \ref{propheat2}, $V|_{x_n=0}\in { B}^{-\frac{1}{q},-\frac{1}{2q}}_{q0}(\Rn\times (0,T))$ with  
\[
\|V|_{x_n=0}\|_{{ B}^{-\frac{1}{q},-\frac{1}{2q}}_{q0}(\Rn\times (0,T))}
\leq  T^{\frac{\al_1}2 }\|{\mathcal F}\|_{L^p(\R_+\times {\mathbb R}_+)}.\]
Since $R_i$ is $L^q$ bounded (see \eqref{propriesz2}),$R'(g_n-v_n|_{x_n=0}-V_n|_{x_n=0})\in  {B}^{-\frac{1}{q},-\frac{1}{2q}}_q(\Rn\times (0,T)) $ with
\begin{align*}
&\|R'(g_n-v_n|_{x_n=0}-V_n|_{x_n=0})\|_{ {B}^{-\frac{1}{q},-\frac{1}{2q}}_q(\Rn\times (0,T))}\\
&\leq\|(g_n-v_n|_{x_n=0}-V_n|_{x_n=0})\|_{ {B}^{-\frac{1}{q},-\frac{1}{2q}}_q(\Rn\times (0,T))}. 
\end{align*}
In the end, we conclude that   $G'=g'-v'|_{x_n=0}-V'|_{x_n=0}-R'(g_n-v_n|_{x_n=0}-V_n|_{x_n=0})\in B^{\-\frac{1}{q},-\frac{1}{2q}}_{q}(\Rn\times (0,T))$
with
\begin{align*}
\|G'\|_{{B}^{-\frac{1}{q},-\frac{1}{2q}}_q(\Rn\times (0,T))}
&\leq c \big( \|g\|_{ B_{q}^{-\frac{1}{q},-\frac{1}{2q}}(\Rn\times (0,T))  }
+\max\{1,T^{\frac{1}{2q}}\}\|h\|_{ {B}_q^{ -\frac2q}(\R_+)}+T^{\frac{\al_1}2 }\|{\mathcal F}\|_{L^p(\R_+\times {\mathbb R}_+)} \big).
\end{align*}

Recall that if $q>3$, then $B_{q}^{-\frac{1}{q},-\frac{1}{2q}}(\Rn\times (0,T)))=B_{q0}^{-\frac{1}{q},-\frac{1}{2q}}(\Rn\times (0,T)))$. Hence we conclude that $G\in B_{q0}^{-\frac{1}{q},-\frac{1}{2q}}(\Rn\times (0,T))$.

If $1<q\leq 3$, then from the fact that $V|_{x_n=0}\in B^{-\frac{1}{q},-\frac{1}{2q}}_{q0}(\Rn\times (0,T))$ and from the hypothesis $g-v|_{x_n=0}=g-\Gamma_t*_x\tilde{h}|_{x_n=0}\in B^{-\frac{1}{q},-\frac{1}{2q}}_{q0}(\Rn\times (0,T))$,  we still conclude that $G\in B_{q0}^{-\frac{1}{q},-\frac{1}{2q}}(\Rn\times (0,T))$ with some modification that
\begin{align*}
\|G'\|_{{B}^{-\frac{1}{q},-\frac{1}{2q}}_{q0}(\Rn\times (0,T))}
&\leq c \big( \|g- \Gamma_t*_x\tilde{h}|_{x_n=0} \|_{B_{q0}^{-\frac{1}{q},-\frac{1}{2q}}(\Rn\times (0,T))}
+T^{\frac{\al_1}2 }\|{\mathcal F}\|_{L^p(\R_+\times {\mathbb R}_+)}\big).
\end{align*}

6)
Finally, applying Theorem \ref{Rn-1} to the fact that $G=(G',0)\in B_{q0}^{-\frac{1}{q},-\frac{1}{2q}}(\Rn\times (0,T))$, we conclude that $w\in L^q(\R_+\times (0,T))$
with
\begin{align*}
\| w\|_{ L^q({\mathbb R}^n_+\times (0,T))}
 &\leq c\max\{1,T^{\frac{1}{2q}}\}\|G\|_{ B^{-\frac{1}{q},-\frac{1}{2q}}_{q0}({\mathbb
R}^{n-1} \times (0,T))}.
 \end{align*}

%
%

This completes the proof of the estimate of the solution in Theorem \ref{thm-stokes}.


\subsection{Regularity}
\setcounter{equation}{0}
\label{regularity-stokes}

Using the estimate of the heat kernel $\Gamma_t$, direct computation of $v=\Gamma_t*\tilde{h}$ leads to the estimate that
\[
\|\nabla  v\|_{L^q(\R)}\leq ct^{-\frac{1}{2}}\|h\|_{L^q(\R_+)}\]
and
\[
\|\nabla  v\|_{L^q(\R)}\leq ct^{-\frac{1}{2}-\frac{1}{q}}\|h\|_{\dot{B}^{-\frac{2}{q}}_q(\R_+)},\]
According to Lemma \ref{propheat2}, $ V_j(x,t)=-\int^t_{0}\int_{\R}D_{y_k}\Gamma(x-y,t-s)\Big(\delta_{ij}\tilde{F}_{ki}+R_iR_j\tilde{F}_{ki}\Big)(y,s)dyds\in W^{1,\frac{1}{2}}_p(\R\times (0,T))$.
Using the estimate of the Poisson kernel $D_{x_n}N(x'-y',x_n)$, direct computation of $\phi=\int_{\Rn}N(x'-y',x_n)\big(g_n(y',t)-v_n(y',0,t)-V_n(y',0,t)\big)dy'$ leads to the estimate that
\[
\|\nabla^2 \phi\|_{L^q(\Rn\times (0,T))}\leq cx_n^{-1}\|g_n-v_n|_{x_n=0}-V_n|_{x_n=0}\|_{L^q(\Rn\times (0,T))}\]
and
\[
\|\nabla^2 \phi\|_{L^q(\Rn\times (0,T))}\leq cx_n^{-1-\frac{1}{q}}\|g_n-v_n|_{x_n=0}-V_n|_{x_n=0}\|_{L^q(0,T;\dot{B}^{-\frac{1}{q}}(\Rn))}.\]
Finally, recall Remark \ref{regularity-zero} that
\[
\|D_xw(\cdot, x_n,t)\|_{L^q(\Rn)}\leq ct^{\frac{1}{2}}x_n^{-2}\|G\|_{L^q(\Rn \times (0,T))}.
\]
and
\[
\|D_xw(\cdot, x_n,t)\|_{L^q(\Rn)}\leq ct^{\frac{1}{2}}x_n^{-2-\frac{1}{q}}\|G\|_{\dot{B}^{-\frac{1}{q},-\frac{1}{2q}}_{q0}(\Rn \times (0,T))},
\]
where $G=(G',0), G'=g'-v'|_{x_n=0}-V'|_{x_n=0}-R'(g_n-v_n|_{x_n=0}-V_n|_{x_n=0})$.
Therefore, we conclude that $\nabla u=\nabla (v+V+\nabla \phi+w)\in L^p_{loc}(\R_+\times (0,T)).$

\subsection{Uniqueness}

Suppose that  $(u_1,p_1)$ and $(u_2,p_2)$ are   very weak
 solutions of  the Stokes equations \eqref{maineq-stokes} in the class $ L^q({\mathbb R}^n_+\times (0,T))$,
 then $u_1-u_2$ satisfies the variational formulation
\[
\int^T_0\int_{{\mathbb R}^n_+}(u_1-u_2)\cdot(-\phi_t-\Delta \phi+\nabla
\pi)dxdt=0\] for any $\phi\in C^\infty_{0,\sigma}({\mathbb
R}^n_+\times [0,T))= \{\phi \in C_0^\infty ({\mathbb R}^n_+ \times
(0,T)) \, | \, div \, \phi(\cdot, t) =0 \quad \mbox{for all} \quad t
\in (0,T) \} $. Since $\{-\phi_t-\Delta \phi+\nabla \pi:\ \phi\in
C^\infty_{0,\sigma}({\mathbb R}^n_+\times [0,T))\} $ is dense in
$L^{\frac{q}{q-1}}({\mathbb R}^n_+\times [0,T))$, we conclude that
$u_1-u_2=0$ a.e. in ${\mathbb R}^n_+\times (0,T)$. Therefore, the
uniqueness of the solution of the Stokes system \eqref{maineq-stokes} holds
in the  class $ L^q({\mathbb R}^n_+\times (0,T))$.

\section{Nonlinear problem}

\label{nonlinear}
\setcounter{equation}{0}

In this section we would like to give a proof of Theorem \ref{thm3}. For the purpose of it, we  construct approximate solutions and then
derive uniform convergence in  $L^q(\R_+\times (0,T))$. For the uniform estimates,
 bilinear estimates should be preceded.

Choose $p=\frac{q}{2}$. Then $\al_1=1-(n+2)(\frac{1}{p}-\frac{1}{q})>0$  for any $q>n+2$ and
\begin{equation}
\label{bilinear1}
\|(u\otimes v)\|_{ L^p(\R\times (0,T))}
\leq c\|u\|_{L^q(\R\times (0,T))}\|v\|_{L^q (\R\times (0,T))}\end{equation}
for any $u,v\in L^q(\R\times (0,T)).$

\subsection{\bf Proof of Theorem \ref{thm3}}

In this section we would like to construct a solution of the Navier-Stokes equations \eqref{maineq2}.

\subsubsection{Approximating solution}


Let $(u^1,p^1)$ be the solution of the equations
\begin{align}
\begin{array}{l}\vspace{2mm}
u^1_t - \De u^1 + \na p^1 =0, \qquad div \, u^1 =0, \mbox{ in }
 \R_+\times (0,T),\\
\hspace{30mm}u^1|_{t=0}= h, \qquad  u^1|_{x_n =0} = g.
\end{array}
\end{align}
Let $m\geq 1$.
After obtaining $(u^1,p^1),\cdots, (u^m,p^m)$ construct $(u^{m+1}, p^{m+1})$ which satisfies the equations
\begin{align}
\label{maineq5}
\begin{array}{l}\vspace{2mm}
u^{m+1}_t - \De u^{m+1} + \na p^{m+1} =f^m, \qquad div \, u^{m+1} =0, \mbox{ in }
 \R_+\times (0,T),\\
\hspace{30mm}u^{m+1}|_{t=0}= h, \qquad  u^{m+1}|_{x_n =0} = g,
\end{array}
\end{align}
where $f^m=-\mbox{div}(u^m\otimes u^m)$.

\subsubsection{Uniform boundedness}
Let $q>n+2$.
By the result of Theorem \ref{thm-stokes}, we have

\begin{align}
\label{uc1-1}
\notag\| u^{1}\|_{L^{q}({\mathbb R}^n_+\times (0,T))}
& \leq c_1 \big( \max\{1,T^{\frac{1}{q}}\}\|h\|_{ B^{-\frac{2}{q}}_{q}({\mathbb
R}^{n}_+)}
+\max\{1,T^{\frac{1}{2q}}\}\|g\|_{  B^{-\frac{1}{q},-\frac{1}{2q}}_{q}({\mathbb
R}^{n-1} \times (0,T))}\\
&\quad  +\|g_n\|_{ L^q(0,T;\dot{B}^{-\frac{1}{q}}_q(\Rn))} \big),
\end{align}
According to    bilinear estimate \ref{bilinear1}, choosing $p=\frac{q}{2}$, we have
\[
\|(u^m\otimes u^m)\|_{L^p(\R\times (0,T))}
\leq c
\|u^m\|_{L^q(\R\times (0,T))}^2.
\]
Hence, we have
\begin{align}
\notag\| u^{m+1}\|_{L^{q}({\mathbb R}^n_+\times (0,T))}
 &\leq
  c_1 \big( \max\{1,T^{\frac{1}{q}}\}\|h\|_{ B^{-\frac{2}{q}}_{q}({\mathbb
R}^{n}_+)}
+\max\{1,T^{\frac{1}{2q}}\}\|g\|_{ B^{-\frac{1}{q},-\frac{1}{2q}}_{q}({\mathbb
R}^{n-1} \times (0,T))}\\
&\quad +\|g_n\|_{ L^q(0,T;\dot{B}^{-\frac{1}{q}}_q(\Rn))}
\label{ucl-2}+
T^{\frac{1}{2}-\frac{n+2}{2q}}\|u^m\|_{ L^{q}({\mathbb
R}^{n}_+ \times (0,T))}^2 \big).
\end{align}

Set
 \begin{align*}
M_0&=\|h\|_{  B^{-\frac{2}{q}}_{q}({\mathbb
R}^{n}_+)}+
\|g\|_{  B^{-\frac{1}{q},-\frac{1}{2q}}_{q}({\mathbb
R}^{n-1} \times (0,T))}
 +\|g_n\|_{ L^q(0,T;\dot{B}^{-\frac{1}{q}}_q(\Rn))}.\end{align*}

Choose $M>2c_1M_0$. Then \eqref{uc1-1} leads to the estimate
\[
\|u^1\|_{L^q(\R_+\times (0,T))}\leq c_1M_0<M\mbox{ for }T\leq 1.\]
Under the condition that $\|u^m\|_{L^q(\R_+\times (0,T))}\leq M$,  \eqref{ucl-2} leads to the estimate
\[
\|u^{m+1}\|_{L^q(\R_+\times (0,T))}\leq c_1M_0+
c_1T^{\frac{1}{2}}
 M^2\mbox{ for }T\leq 1.
\]
Choose $0<T\leq \frac{1}{(2c_1M)^2}$, together with the condition $T\leq 1$. Then by the mathematical induction argument we can conclude that
\[
\|u^{m}\|_{L^q(\R_+\times (0,T))}\leq M\mbox{ for all }m=1,2\cdots.
\]

\subsubsection{Uniform convergence}

Let $U^m=u^{m+1}-u^m$ and $P^m=p^{m+1}-p^m$.
Then $U^m$ satisfies the equations
\[
\begin{array}{l}\vspace{2mm}
U^m_t - \De U^m + \na P^m =-\mbox{div}(u^m\otimes U^{m-1}+U^{m-1}\otimes u^{m-1}), \qquad div \, U^{m} =0, \mbox{ in }
 \R_+\times (0,T),\\
\hspace{30mm}U^{m}|_{t=0}= 0, \qquad  U^{m}|_{x_n =0} =0,
\end{array}
\]

By the result of Theorem \ref{thm-stokes},   we have
\begin{align*}
\|U^m\|_{L^q({\mathbb R}^n_+\times (0,T))}
&\leq c_2T^{\frac{1}{2}-\frac{n+2}{2q} }
\|(u^m\otimes U^{m-1}+U^{m-1}\otimes u^{m-1})\|_{L^p(\R_+\times (0,T))}
\\
&\leq c_2T^{\frac{1}{2}-\frac{n+2}{2q} } (\|u^m\|_{L^q(\R_+\times (0,T))}+\|u^{m-1}\|_{L^q(\R_+\times (0,T))})\|U^{m-1}\|_{L^q(\R_+\times (0,T))}.
\end{align*}

Choose $0<T\leq \frac{1}{2c_2M}  $ together with the condition $T \leq    \frac{1}{(2c_1M)^2}$ and $T\leq 1$. Then, the above estimate leads to the
\begin{equation}
\label{m2-1}
\|U^m\|_{L^q({\mathbb R}^n_+\times (0,T))}\leq \frac{1}{2}\|U^{m-1}\|_{L^q(\R_+\times (0,T))}.
\end{equation}
\eqref{m2-1} implies that the infinite series  $\sum_{k=1}^\infty U^k$ converges in  $L^q(\R_+\times (0,T))$.
Again it means that   $u^m=u^1+\sum_{k=1}^mU^{k}$ converges to $u^1+\sum_{k=1}^\infty U^{k}$ in  $L^q(\R_+\times (0,T))$.
Set $u:=u^1+\sum_{k=1}^\infty U^{k}.$

\subsection{Existence and regularity}

Let $u$ be the same one constructed by the previous section.
In this section, we will show that $u$ satisfies weak formulation of Navier-Stokes equations, that is, $u$ is a weak solution of Navier-Stokes equations with appropriate distribution $p$.

Let $\Phi\in C^\infty_{0}(\overline{\R}_+\times (0,T))$ with $\mbox{div }\Phi=0$ and $\Phi|_{x_n=0}=0$.
Observe that
\[
-\int^T_0\int_{\R_+} u^{m+1}\cdot \Delta\Phi dxdt=\int^T_0\int_{\R_+}u^{m+1}\cdot \Phi_t+(u^m\otimes u^m): \nabla \Phi dxdt+<h,\Phi(\cdot,0)>_{\R_+}-<g,\frac{\partial\Phi}{\partial x_n}>_{\Rn\times {\mathbb R}_+}.\]
Now send $m$ to the infinity, then, since $u^m\rightarrow u$ in $L^q(\R_+\times (0,T))$, we have
\begin{equation}
\label{w1}
-\int^T_0\int_{\R_+}u\cdot \Delta \Phi dxdt=\int^T_0\int_{\R_+}u\cdot \Phi_t+(u\otimes u): \nabla\Phi dxdt+<h,\Phi(\cdot,0)>_{\R_+}-<g,\frac{\partial\Phi}{\partial x_n}>_{\Rn\times {\mathbb R}_+}.\end{equation}

Since $u$ satisfies the Stokes equations \eqref{maineq-stokes} with  $f=-\mbox{div}{\mathcal F}=-\mbox{div}(u\otimes u)$, $u$ can be decomposed by $u=v+V+\nabla \phi+w$, where
\begin{align*}
v(x,t)&=\int_{\R}\Gamma(x-y,t)\tilde{h}(y)dy,\\
V_j(x,t)&=\int^t_{0}\int_{\R}D_{x_k}\Gamma(x-y,t-s)[\delta_{ij}\tilde{u}_{k}\tilde{u}_i+R_iR_j\tilde{u}_{k}\tilde{u}_{i}](y,s)dyds,\\
\phi(x,t)&=2\int_{\Rn}N(x'-y',x_n)\Big(g_n(y',t)-v_n(y',0,t)-V_n(y',0,t)\Big)dy'\\
 w_i&= \sum_{j=1}^{n-1}\int_0^t \int_{\Rn} K_{ij}( x'-y',x_n,t-s)G_j(y',s) dy'ds,
  \end{align*}
  for $G=(g' -V'|_{x_n=0} -v'|_{x_n=0} - R'(g_n -v_n|_{x_n=0}-V_n|_{x_n=0},0).$
Observe that $\delta_{ij}\tilde{u}_{k}\tilde{u}_i+R_iR_j\tilde{u}_{k}\tilde{u}_{i}\in L^{\frac{q}{2}}(\R_+\times (0,T))$. According to Lemma \ref{propheat2}, $ V_j(x,t)
\in W^{1,\frac{1}{2}}_{\frac{q}{2}}(\R\times (0,T))$.
On the other hand, by the same  argument in section \ref{regularity-stokes},
we have
\[
\|\nabla  v\|_{L^q(\R)}\leq ct^{-\frac{1}{2}}\|h\|_{L^q(\R_+)}\]
\[
\|\nabla  v\|_{L^q(\R)}\leq ct^{-\frac{1}{2}-\frac{1}{q}}\|h\|_{\dot{B}^{-\frac{2}{q}}_q(\R_+)},\]
\[
\|\nabla^2 \phi\|_{L^q(\Rn\times (0,T))}\leq cx_n^{-1}\|g_n-v_n|_{x_n=0}-V_n|_{x_n=0}\|_{L^q(\Rn\times (0,T))}\]
\[
\|\nabla^2 \phi\|_{L^q(\Rn\times (0,T))}\leq cx_n^{-1-\frac{1}{q}}\|g_n-v_n|_{x_n=0}-V_n|_{x_n=0}\|_{L^q(0,T;\dot{B}^{-\frac{1}{q}}(\Rn))},\]
\[
\|D_xw(\cdot, x_n,t)\|_{L^q(\Rn)}\leq ct^{\frac{1}{2}}x_n^{-2}\|G\|_{L^q(\Rn \times (0,T))},
\]
\[
\|D_xw(\cdot, x_n,t)\|_{L^q(\Rn)}\leq ct^{\frac{1}{2}}x_n^{-2-\frac{1}{q}}\|G\|_{\dot{B}^{-\frac{1}{q},-\frac{1}{2q}}_q(\Rn \times (0,T))}.
\]
and
\[
\|D_xw(\cdot, x_n,t)\|_{L^q(\Rn)}\leq ct^{\frac{1}{2}}x_n^{-2}\|G\|_{L^q(\Rn \times (0,T))}.
\]
Therefore, we conclude that $\nabla u=\nabla (v+V+\nabla \phi+w)\in L^{\frac{q}{2}}_{loc}(\R_+\times (0,T)).$
%
%
 This leads to the conclusion that $u$  is a weak solution of the  Navier-Stokes equations \eqref{maineq2}.

\subsection{Uniqueness}

Let  $ v\in L^q(\R_+\times (0,T))$ be  another solution of Naiver-Stokes equations \eqref{maineq2} with pressure $q$. Then
 $u-v$ satisfies the equations
\begin{align*}
(u-v)_t - \De (u-v) + \na (p-q)& =-\mbox{div}(u\otimes (u-v)+(u-v)\otimes v)\mbox{ in }
 \R_+\times (0,T), \\
 div \, (u-v)& =0,
 \mbox{ in }\R_+\times (0,T),\\
 (u-v)|_{t=0}= 0, &\quad (u-v)|_{x_n =0} =0.
\end{align*}
Applying  Theorem \ref{thm-stokes} to the above Stokes equations for $u-v$, then  we have
\begin{align*}
\| u-v\|_{L^q({\mathbb R}^n_+\times (0,T_1))}
 \leq cT_1^{\frac{1}{2}-\frac{n+2}{2q}} \|u\otimes (u-v)+(u-v)\otimes v \|_{L^p(\R_+\times (0,T_1))}\\
 \leq c_3T_1^{\frac{1}{2}-\frac{n+2}{2q}} ( \|u\|_{L^q({\mathbb R}^n_+\times (0,T_1))}+\|v\|_{L^q({\mathbb R}^n_+\times (0,T_1))})\| u-v\|_{L^q({\mathbb R}^n_+\times (0,T_1))}, \quad  T_1\leq T.
 \end{align*}
 If we take $T_1\leq \frac{1}{4c_3^2(\|u\|_{L^q({\mathbb R}^n_+\times (0,T))}+\|v\|_{L^q({\mathbb R}^n_+\times (0,T))}+1)^2}$ together with $T_1\leq 1$, then the above inequality leads to the conclusion that
 \[
 \| u-v\|_{L^q({\mathbb R}^n_+\times (0,T_1))}=0\mbox{ that is, }u\equiv v\mbox{ in }\R_+\times (0,T_1).\]
 By the same argument, we can show that
\[
 \| u-v\|_{L^q({\mathbb R}^n_+\times (T_1,2T_1))}=0\mbox{ that is, }u\equiv v\mbox{ in }\R_+\times (T_1,2T_1).\]
After iterating this procedure  finitely many times, we obtain  the conclusion that $u=v$ in $\R_+\times (0,T)$.

\appendix
\setcounter{equation}{0}

\section{Proof of Lemma \ref{lemma0115}}

The following is well known estimates(see  section 4.3 in \cite{lady?} for the reference):
\begin{align}
\label{T_11}\|T_1f\|_{\dot{W}^{2,1}_q(\R \times {\mathbb R})}&\leq c\|f\|_{L^q(\R\times {\mathbb R})}.
\end{align}
Observe that $T_1^*$ is adjoint operator of $T_1$, since
\begin{align*}
\int_{{\mathbb R}} \int_{\R} T_1\psi(x,t)  dxdt =  \int_{{\mathbb R}} \int_{\R} T_1^* \phi(y,s) \psi(y, s) dyds.
\end{align*}
%
Observe that $D^2_yT_1^*f, \ D_s T_1^*f$ have $L^p$ Fourier multipliers since the Fourier transform of $T_1^*f$  is $
\widehat{T^*_1 f}=\frac{1}{|\xi|^2-i \eta}
\hat{f}(\xi,\eta)$. 
By the 
well known theory for the  multiplier (see \cite{St}) we have
\begin{align}
\label{T^*1}
\|T_1^*f\|_{\dot W^{2,1}_{p}(\R \times {\mathbb R})}    \leq c \|f\|_{L^p(\R \times {\mathbb R})}, \quad   1<p<\infty.
\end{align}
Since $T^*$ is the adjoint operator of $T$, \eqref{T^*1} implies that
\begin{align}
\label{T2}
\| T_1f\|_{L^p(\R \times {\mathbb R})} \leq  c\|f\|_{\dot W^{-2,-1}_p(\R\times {\mathbb R})},
\end{align}
%
and \eqref{T_11} implies that
\begin{align}
\label{T^*2}
\| T_1^*f\|_{L^p(\R \times {\mathbb R})} \leq  c\|f\|_{\dot W^{-2,-1}_p(\R\times {\mathbb R})}.
\end{align}
%
Applying real interpolation theory to \eqref{T2} and \eqref{T_11},
 we  complete the proof of the estimate $T_1f$ in Lemma \ref{lemma0115} for $2>\al>0$. 
 Also, applying real interpolation theory to  \eqref{T^*2} and \eqref{T^*1}, we  complete the proof of the estimate $T_1^*f$ in Lemma \ref{lemma0115}  for $2>\al>0$.



\section{Proof of Lemma \ref{lem-T}}

The following is well known estimates(see  section 4.3 in \cite{lady?} for the reference):
\begin{align}
\label{T_21}\|T_2g\|_{\dot{W}^{2,1}_q(\R_+ \times {\mathbb R})}&\leq c\|g\|_{\dot{B}_q^{1-\frac{1}{q},\frac12-\frac{1}{2q}}(\Rn \times {\mathbb R})}.
\end{align}

Firstly, let us derive the  estimate of $T_2g$.
Observe the identity
\begin{align}\label{1027-1}
\int_{-\infty}^\infty \int_{\R_+} T_2g(x,t) \phi(x,t) dxdt = <g, T_1^*\phi|_{y_n=0}>
\end{align}
holds for $\phi\in C^\infty_0(\R_+\times {\mathbb R})$,
where $ T_1^*\phi(y,s) = \int_s^\infty \int_{\R}  \Ga(x-y, t-s) \phi(x,t) dx dt$ and  $<\cdot,\cdot>$ is the duality pairing between $\dot{B}^{-1-\frac{1}{q},-\frac{1}{2}-\frac{1}{2q}}_q(\Rn\times {\mathbb R})$ and $\dot{B}^{1+\frac{1}{q},\frac{1}{2}+\frac{1}{2q}}_{q'}(\Rn\times {\mathbb R}).$ 
%
From the result of  Lemma \ref{lemma0115}
we have
\begin{align}
\label{point1}
\|T_1^*\phi\|_{\dot W^{2,1}_{q'}(\R \times {\mathbb R})}    \leq c \|\phi\|_{L^{q'}(\R \times {\mathbb R}) }=\|\phi\|_{L^{q'}(\R_+ \times {\mathbb R}) }.
\end{align}
By trace theorem, we have
\begin{align}\label{1027-2}
\|T_1^*\phi|_{y_n=0}\|_{\dot B^{1+\frac1q, \frac{1}{2}+\frac1{2q}  }_{q'}(\Rn \times {\mathbb R})}
\leq c  \|T_1^*\phi\|_{\dot W^{2,1}_{q'}(\R \times {\mathbb R})}
\leq c \|\phi\|_{L^{q'}(\R_+ \times {\mathbb R}) }.
\end{align}
Apply the estimate \eqref{1027-2} to \eqref{1027-1}, we have
\begin{equation}
\label{mathcalT1}
\| T_2g \|_{L^q(\R_+ \times {\mathbb R})}\leq c\|g\|_{\dot B^{-1-\frac1q, -\frac{1}{2}-\frac1{2q}  }_q(\Rn \times {\mathbb R}) }.
\end{equation}
Applying real interpolation theory to \eqref{mathcalT1} and \eqref{T_21},
 we  complete the proof of the estimate of $T_2g$ in  Lemma \ref{lem-T}.

Analogously, we can derive the estimate of $T_2^*g$, 
observing the identity
\begin{align}\label{1027-20}
\int_{-\infty}^\infty \int_{\R_+} T_2^*g(y,s) \phi(y,s) dyds = <g,  T_1\phi|_{x_n=0}>
\end{align}
holds for $\phi\in C^\infty_0(\R_+ \times {\mathbb R})$,
where $ T_1\phi(x,t) = \int_{-\infty}^t \int_{\R}  \Ga(x-y, t-s) \phi(y,s) dy ds$ and  $<\cdot,\cdot>$ is the duality pairing between $\dot{B}^{-1-\frac{1}{q},-\frac{1}{2}-\frac{1}{2q}}_q(\Rn\times {\mathbb R})$ and $\dot{B}^{1+\frac{1}{q},\frac{1}{2}+\frac{1}{2q}}_{q'}(\Rn\times {\mathbb R}).$
By the same procedure as  for the estimate of $T_2f$(We omit the details), we can  obtain the estimate of $T_2^*g$ that
\begin{equation}
\label{mathcalT1-1}
\| T_2^*g \|_{L^q(\R_+ \times {\mathbb R})}\leq c\|g\|_{\dot B^{-1-\frac1q, -\frac{1}{2}-\frac1{2q}  }_q(\Rn \times {\mathbb R}) }.
\end{equation}
Since $D_{y_k}D_{y_l}T_2^*g=T_2^*(D_{y_k}D_{y_l}g)$ for $k,l\neq n$, and $D_sT_2^*g=T_2^*(D_sg)$ we  have
\begin{equation}
\label{mathcalT1-2}
\sum_{k\neq n}\| D_{x_k}^2T_2^*g \|_{L^q(\R_+ \times {\mathbb R})}+\|D_sT_2^*g \|_{L^q(\R_+ \times {\mathbb R})}\leq c\|g\|_{\dot B^{1-\frac1q, \frac{1}{2}-\frac1{2q}  }_q(\Rn \times {\mathbb R}) }.
\end{equation}
Since $D_{x_n}^2 T_2^*g=-D_sT_2^*g-\sum_{neq n}D_{x_k}^2 T_2^*g$, we again have
\begin{equation}
\label{mathcalT1-3}
\| \Delta_y T_2^*g \|_{L^q(\R_+ \times {\mathbb R})}+\|D_sT_2^*g \|_{L^q(\R_+ \times {\mathbb R})}\leq c\|g\|_{\dot B^{1-\frac1q, \frac{1}{2}-\frac1{2q}  }_q(\Rn \times {\mathbb R}) }.
\end{equation}
By the well known elliptic theory $T_2^*g|_{y_n=0}=0$ implies that
\[
\|D_y^2T_2^*g \|_{L^q(\R_+ \times {\mathbb R})}\leq c\| \Delta_y T_2^*g \|_{L^q(\R_+ \times {\mathbb R})}.
\]
Combining all the above estiamtes we conclude that
\begin{equation}
\label{mathcalT1-4}
\| T_2^*g \|_{W^{2,1}_q(\R_+ \times {\mathbb R})}\leq c\|g\|_{\dot B^{1-\frac1q, \frac{1}{2}-\frac1{2q}  }_q(\Rn \times {\mathbb R}) }.
\end{equation}

Applying real interpolation theory to \eqref{mathcalT1-1} and \eqref{mathcalT1-4}, we complete the proof of the estimate of $T_2^*g$ in \ref{lem-T} for $0<\al<2$.

\section{Proof of Lemma \ref{proheat1} }

$\bullet$ First we would like to derive the estimate of $T_0h=\Gamma_t*h$.

Let us consider the case $h\in \dot B^{-\frac{2}{q}}_q(\R)$.
%
Observe  the identity
$
\int^\infty_{0}\int_{\R}T_0h(x,t)\phi(x,t)dx dt=<h,T_1^*\phi|_{s=0}>$ holds for  $\phi\in C_0^\infty(\R\times {\mathbb R})$,
where $T_1^*\phi(y,s)=\int^\infty_{s}\int_{\R}\Gamma(x-y,t-s)\phi(x,t)dxdt,$ and $<\cdot,\cdot>$ is the duality pairing between $\dot{B}^{-\frac{2}{q}}_q(\R)$ and $\dot{B}^{\frac{2}{q}}_{q'}(\R).$
From the result of Lemma \ref{lemma0115}, we have
\[
\|T_1^*\phi\|_{\dot W^{2,1}_{q'}(\R \times {\mathbb R})}
\leq c \|\phi\|_{L^{q'}(\R \times {\mathbb R}) }.\]
%
%
%
By trace theorem this implies that
\[
\|T_1^*\phi|_{s=0}\|_{\dot{B}_{q'}^{2-\frac{2}{q'}}(\R)}\leq c\|T^*_1\phi\|_{\dot W^{2,1}_{q'}(\R \times {\mathbb R})}
\leq c \|\phi\|_{L^{q'}(\R \times (0, \infty)) }.\]
Hence, we have \[
<h,T_1^*\phi|_{s=0}>\leq c\|h\|_{\dot{B}_q^{-\frac{2}{q}}(\R)}\|T_1^*\phi\|_{\dot{B}^{2-\frac{2}{q'}}(\R)}\leq c\|h\|_{\dot{B}_q^{-\frac{2}{q}}(\R)}\|\phi\|_{L^{q'}(\R \times {\mathbb R}) } .\]
Again this leads to the conclusion 
%
that
\begin{equation}
\label{it1}
\|T_0h\|_{L^q(\R\times {\mathbb R})}\leq c\|h\|_{\dot{B}^{-\frac{2}{q}}_q(\R)}.
\end{equation}
%
%
%
On the other hand, by Young's theorem we have
\begin{equation}
\label{heat0}
\|T_0h\|_{L^q(\R\times (0,T))}\leq cT^{\frac{1}{q}}\|h\|_{L^q(\R)}.
\end{equation}
From the fact that  ${B}^{s}_q(\R)=\dot{B}^{s}_q(\R)+L^q(\R)$ for $s < 0$,  \eqref{it1} and \eqref{heat0} imply that
\begin{equation}
\label{heat1}
\|T_0h\|_{L^q(\R\times (0,T))}\leq c\max\{1,T^{\frac{1}{q}}\}\|h\|_{{B}^{-\frac{2}{q}}_q(\R)}.
\end{equation}

%

$\bullet$
Now, we will derive the estimate of $T_0h|_{x_n=0}=\Gamma_t*f|_{x_n=0}$.


Let $h\in \dot{B}^{-\frac{2}{q}}_q(\R)$.
Observe the identity
\begin{equation}
\label{holder1}
<T_0h,\phi>=<h, T_2^*\phi|_{s=0}>,\end{equation} holds for
any  $\phi \in C^\infty_0(\Rn\times {\mathbb R})$,
where $T_2^*\phi(y,s)=\int^\infty_s\int_{\Rn}\Gamma(x'-y',y_n,t-s)\phi(x',t)dx'dt$ and $<\cdot,\cdot>$ is the duality pairing between $\dot{B}^{-\frac{2}{q}}_q(\R)$ and $\dot{B}^{\frac{2}{q}}_{q'}(\R)$ .
From the result of Lemma \ref{lem-T},
$
T_2^*\phi\in \dot{W}^{2,1}_{q'}(\R_+\times {\mathbb R})$  with
\[
\|T_2^*\phi\|_{ \dot{W}^{2,1}_{q'}(\R_+\times {\mathbb R})}\leq c\|\phi\|_{\dot{B}^{\frac{1}{q},\frac{1}{2q}}_{q'}(\Rn\times {\mathbb R})}
.\]
By Trace theorem, this implies that
$T_2^*\phi\Big|_{s=0}\in \dot{B}^{\frac{2}{q}}_{q'}(\R_+ )$ with
\[
\|T_2^*\phi\Big|_{s=0}\|_{ \dot{B}^{\frac{2}{q}}_{q'}(\R_+ )}
   \leq c\|\phi\|_{\dot{B}^{\frac{1}{q},\frac{1}{2q}}_{q'}(\Rn\times {\mathbb R})}.\]
Hence
\[
<h, T_2^*\phi\Big|_{s=0}>\leq c\|h\|_{\dot{B}^{-\frac{2}{q}}_q(\R)}\|\phi\|_{\dot{B}^{\frac{1}{q},\frac{1}{2q}}_{q'}(\Rn\times {\mathbb R})}\]
Applying the above estimate to \eqref{holder1}, we conclude that
$T_0h|_{x_n=0}\in  \dot{B}^{\-\frac{1}{q},-\frac{1}{2q}}_{q}(\Rn\times {\mathbb R})$ with
\begin{align}
\label{y1}
\|T_0h|_{x_n=0}\|_{\dot{B}^{-\frac{1}{q},-\frac{1}{2q}}_{q}(\Rn\times {\mathbb R})}\leq c\|h\|_{\dot{B}^{-\frac{2}{q}}_q(\R)}.\end{align}
On the other hand, by  Young's inequality we have
\begin{align}
\label{y2}
\| T_0h|_{x_n=0}\|_{L^q(\Rn \times (0, T))} \leq c T^{\frac1{2q}} \| h\|_{L^q(\R)}.
\end{align}
Recall the fact that $B^{s,\frac{s}{2}}_q(\Omega\times (0,T))=\dot B^{s,\frac{s}{2}}_q(\Omega\times (0,T))+L^q(\Omega\times (0,T))$ for $s<0$. Combining \eqref{y1} and \eqref{y2}, we have
\begin{align}
\label{y22}
\| T_0h|_{x_n=0}\|_{B^{ -\frac1q,  -\frac1{2q}}_{q} (\Rn \times (0, T))} \leq c \max\{1,T^{\frac1{2q}}\} \| h\|_{B_q^{ -\frac2 q}(\R)}.
\end{align}

\section{Proof of Lemma \ref{propheat2} }
Note that $u=D_xT_1\tilde{f}.$
By Lemma \ref{lemma0115}, the following estimate holds
\[
\|{u}\|_{\dot W^{1,\frac{1}{2}}_p(\R\times {\mathbb R})}\leq \|T_1\tilde{f}\|_{\dot W^{2,1}_p(\R\times {\mathbb R})}\leq c\|\tilde{ f}\|_{L^p(\R\times {\mathbb R}))}.
\]
On the other hand, by Young's inequality we have
\begin{equation}
\label{itra2}
\|u\|_{L^{q}(\R\times (0,T))}\leq cT^{\frac{\al_1}{2} }\|{ f}\|_{L^p(\R\times (0,T))},
\end{equation}
where $\al_1=1-(n+2)(\frac{1}{p}-\frac{1}{q_1})>0$.

 Now we will derive the estimate of $u|_{x_n=0}$ in $B^{-\frac{1}{q},-\frac{1}{2q}}_{q0}(\Rn\times (0,T))$.
Choose $q_1=\frac{(n+1)q}{n+2}$ so that $q_1<q$  and $(\be_1:=)1-(n+2)(\frac{1}{p}-\frac{1}{q_1})>\frac{1}{q_1}$. By Young's inequality
\begin{equation}
\label{bdit2}
\|u|_{x_n=0}\|_{L^{q_1}(\Rn \times (0,T))}\leq cT^{\frac{\be_1}{2}-\frac{1}{2q_1} }\|{ f}\|_{L^p(\R\times (0,T))}.
\end{equation}
Observe that $L^{q_1}(\Rn \times (0,T))\subset B^{-\frac{1}{q},-\frac{1}{2q}}_{q0}(\Rn\times (0,T))$ for $-\frac{n+1}{q_1} = -\frac1q -\frac{n+1}{q}$.
Note that
\begin{align*}
\frac{\be_1}{2}-\frac{1}{2q_1}
= \frac12  ( 1-(n+2)(\frac{1}{p}-\frac{1}{q_1}) ) -\frac1{2q_1}
= \frac12  ( 1-(n+2)(\frac{1}{p}-\frac{1}{q}) )
 = \frac{\al_1}2.
\end{align*}
This completes the proof of Lemma \ref{propheat2}.


\begin{thebibliography}{10}



\bibitem{fernandes} M.F. de Almeida and L.C.F. Ferreira,{\it On the Navier-Stokes equations in the half-space with initial and boundary rough data in Morrey spaces}, J. Differential Equations  254,  no. 3, 1548-1570(2013).

\bibitem{amman-anisotropic} H. Amann, {\it
Anisotropic function spaces and maximal regularity for parabolic problems. Part 1.
Function spaces,}  Jind$\check{\rm r}$ich Ne$\check{\rm c}$as Center for Mathematical Modeling Lecture Notes, 6. Matfyzpress, Prague, vi+141(2009).

\bibitem{amann}H. Amann,{\it On the strong solvability of the Navier-Stokes equations}, J. Math. Fluid Mech.  2,  no. 1, 16-98 (2000).

\bibitem{amann1} H. Amann,{\it Navier-Stokes equations with nonhomogeneous Dirichlet data}, J. Nonlinear Math. Phys.  10,  suppl. 1, 1-11(2003).

\bibitem{amann2} H. Amann,{\it Nonhomogeneous Navier-Stokes Equations with Integrable Low-Rregularity data,}  Nonlinear problems in mathematical physics and related topics, II,  1–28, Int. Math. Ser. (N. Y.), 2, Kluwer/Plenum, New York, 2002.


\bibitem{BL}  J. Bergh and J. L$\ddot{\rm o}$fstr$\ddot{\rm o}$m,{ Interpolation Sspaces. An Introduction}, Grundlehren der Mathematischen Wissenschaften, No. 223. Springer-Verlag, Berlin-New York, 1976.




\bibitem{cannone} M. Cannone, F. Planchon and M.Schonbek, {\it Strong solutions to the incompressible Navier-Stokes equations in the half-space}, Comm. Partial Differential Equations  25,  no. 5-6, 903-924(2000).












\bibitem{CC}
T. Chang and H.J. Choe,{\it  Maximum modulus estimate for the solution of the Stokes equations,}
J. Differential Equations  254,  no. 7, 2682-2704(2013).


\bibitem{fabes} E.B. Fabes, B.F. Jones and N.M. Rivi$\grave{\rm e}$re, {\it The initial value problem for the Navier-Stokes equations with data in $L^p$,} Arch.Ration.Mech.Anal.45,222-240(1972).







%


























\bibitem{farwig4} R. Farwig and H. Kozono, {\it Weak solutions of the Navier-Stokes equations with non-zero boundary values in an exterior domain satisfying the strong energy inequality}, J. Differential Equations  256,  no. 7, 2633-2658(2014).

%


  \bibitem{farwig6} R. Farwig,G.P. Galdi and H. Sohr, Very Weak Solutions of Stationary and Instationary Navier-Stokes Equations with Nonhomogeneous Data,.  Nonlinear elliptic and parabolic problems,  113-136, Progr. Nonlinear Differential Equations Appl., 64, Birkh$\ddot{\rm a}$user, Basel, 2005.


\bibitem{farwig2} R. Farwig, H. Kozono, and H. Sohr, {\it Very weak solutions of the Navier-Stokes equations in exterior domains with nonhomogeneous data}, J. Math. Soc. Japan  59,  no. 1, 127-150(2007).

\bibitem{farwig3} R. Farwig, H. Kozono, and H. Sohr, {\it Global weak solutions of the Navier-Stokes equations with nonhomogeneous boundary data and divergence}, Rend. Semin. Mat. Univ. Padova  125, 51-70(2011).



\bibitem{giga2} Y. Giga,{\it Solutions for semilinear parabolic equations in $L^p$ and regularity of weak solutions of the Navier-Stokes system}, J.Differential Equations 62, no. 2,186-212(1986).


    \bibitem{giga3} Y. Giga and T. Miyakawa,{\it Navier-Stokes flow in ${\mathbb R}^3$ with measures as initial velocity and Morrey spaces}, Comm.Partial Differential Equations 14, no. 5, 577-618(1989).

\bibitem{giga}
M. Giga,Y. Giga and H. Sohr, {\it $L^p$   estimates for the Stokes system,}  Functional analysis and related topics, 1991 (Kyoto),  55–67, Lecture Notes in Math., 1540, Springer, Berlin(1993).

 \bibitem{giga1}
 Y. Giga and H. Sohr, {\it Abstract $L^p$   estimates for the Cauchy problem with applications to the Navier-Stokes equations in exterior domains}, J. Funct. Anal.  102,  no. 1, 72-94(1991).


%





%
%
%
%


\bibitem{grubb1} G. Grubb, {\it Nonhomogeneous time-dependent Navier-Stokes problems in $L_p$   Sobolev spaces,} Differential Integral Equations  8,  no. 5, 1013-1046(1995).
%

\bibitem{grubb3} G. Grubb, {\it Nonhomogeneous Dirichlet Navier-Stokes problems in low regularity $L_p$   Sobolev spaces,} J. Math. Fluid Mech.  3,  no. 1, 57-81(2001).


\bibitem{grubb} G. Grubb and V.A. Solonnikov, {\it  Boundary value problems for the nonstationary Navier-Stokes equations treated by pseudo-differential methods}, Math. Scand.  69,  no. 2, 217-290 (1992).





\bibitem{iftimie} D. Iftimie, {\it The resolution of the Navier-Stoes equations in anisotropic spaces}, Revista Mate$\acute{\rm a}$tica Iberoamericana Vol.15, No. 1, 1-35(1999).



\bibitem{kato1} T. Kato, {\it Strong solutions of the Navier-Stokes equation in Morrey spaces,} Bol.Soc.Brasil.Mat.(N.S.) 22, no. 2, 127-155(1992)

\bibitem{kato2} T. Kato, {\it Strong $L^p$ solutions of the Navier-Stokes equation in $\R$, with applications to weak solutios,} Math.Z.187, no.4, 471-480(1984).



\bibitem{kato} T. Kato and G. Ponce,{\it  Well-posedness of the Euler and Navier-Stokes equations in the Lebesgue spaces $L^p_s (R_2 )$,}
  Rev. Mat. Iberoamericana  2,  no. 1-2, 73-88(1986).

\bibitem{koch} H. Koch and D. Tataru,{\it Well-posedness for the Navier-Stokes equations}, Adv. Math.  157,  no. 1, 22-35(2001).



\bibitem{KS1} H. Koch and V. A.  Solonnikov,
{\it $L_p$-Estimates for a solution to the nonstationary Stokes
equations,} Journal of Mathematical Sciences, Vol. 106, No.3,
3042-3072(2001).

\bibitem{KS2} H.  Koch and V. A.  Solonnikov, {\it $L_q$-estimates of the first-order derivatives of solutions to the nonstationary Stokes
problem, Nonlinear problems in mathematical physics and related topics, I},  Int. Math. Ser. (N. Y.), 1, Kluwer/Plenum, New York,203-218(2002).





\bibitem{kozono1} H. Kozono,{\it Global $L^n$-solution and its decay property for the Navier-Stokes equations in half-space $R^n_+ $,} J. Differential Equations  79,  no. 1, 79-88(1989).


\bibitem{kozono2}H. Kozono, and M. Yamazaki, {\it Semilinear heat equations and the Navier-Stokes equation with distributions in new function spaces as initial data}, Comm. Partial Differential Equations  19,  no. 5-6, 959-1014(1994).



\bibitem{lady?}
O.A. Lady$\check{\rm z}$enskaja, V.A. Solonnikov and N.N. Ura${\rm l}ʹ$ceva, {\rm Linear and Quasilinear Equations of Parabolic Type}, (Russian) Translated from the Russian by S. Smith. Translations of Mathematical Monographs, Vol. 23 American Mathematical Society, Providence, R.I. 1968.

\bibitem{lemarie} P.G. Lemarié-Rieusset,{\it The Navier-Stokes equations in the critical Morrey-Campanato space}, Rev. Mat. Iberoam.  23,  no. 3, 897-930(2007).

\bibitem{lewis} J.E. Lewis,{\it The initial-boundary value problem for the Navier-Stokes equations with data in $L^p$,} Indiana Univ. Math. J.  22, 739-761(1972/73).




%



\bibitem{raymond1} J.-P. Raymond,{\it  Stokes and Navier-Stokes equations with nonhomogeneous boundary conditions}, Ann. Inst. H. Poincaré Anal. Non Linéaire  24,  no. 6, 921-951(2007).
















%


\bibitem{sol1}V.A. Solonnikov,{\it Estimates of the solutions of the nonstationary Navier-Stokes system,}
Boundary value problems of mathematical physics and related questions in the theory of functions, 7.
Zap. Naučn. Sem. LOMI.  38, 153-231(1973).


%



%


\bibitem{Sol-1} V.A. Solonnikov,{\it $L_p$-estimates for solutions to the initial boundary-value problem for the generalized Stokes system in a bounded domain,} Function theory and partial differential equations. J. Math. Sci. (New York)  105  (2001),  no. 5, 2448-2484.

\bibitem{Sol-2} V.A. Solonnikov,{\it
Estimates for solutions of the nonstationary Stokes problem in anisotropic Sobolev spaces and estimates for the resolvent of the Stokes operator,} (Russian)  Uspekhi Mat. Nauk  58  (2003),  no. 2(350), 123-156;  translation in  Russian Math. Surveys  58  (2003),  no. 2, 331-365.

%

\bibitem{St} E.M. Stein,{\it Singular Integrals and Differentiability Properties of Functions,}
Princeton Mathematical Series, No. 30 Princeton University Press, Princeton, N.J. 1970.



\bibitem{Triebel} H. Triebel,{\it Theory of Function Spaces,} Monographs in Mathematics, 78. Birkh$\ddot{\rm a}$user Verlag, Basel, 1983.




\bibitem{Triebel2} H. Triebel,{\it  Theory of Function Spaces. III,} Monographs in Mathematics, 100. Birkh$\ddot{\rm a}$user Verlag, Basel, 2006.


\bibitem{voss} K.A. Voss,{\it Self-similar solutions of the Navier-Stokes equation},
Thesis (Ph.D.)–Yale University. 1996.

\bibitem{yamazaki} M.Yamazaki, {\it A quasi-homogeneous version of paradifferential operators, I.Boundedness on spaces of Besov type,}J.Fac.Sci.Tokyo 33,131-174(1986).



%





%
\bibitem{Tr} H. Triebel, {\it Interpolation Theory, Function Spaces, Differential Operators, Second edition},
             Johann Ambrosius Barth, Heidelberg, 1995.
















\end{thebibliography}
\end{document}